\documentclass[a4paper,10pt,reqno]{amsart}
\usepackage{multirow}

\usepackage{amsmath,amsthm,amssymb,enumerate}
\usepackage{epsfig}
\usepackage{amssymb}
\usepackage{amsmath}
\usepackage{amssymb}
\usepackage{amsmath,amsthm}
\usepackage{mathtools}
\usepackage{eqparbox}
\usepackage[latin1]{inputenc}
\usepackage{bm}
\usepackage{bbm}
\usepackage{esint}
\usepackage{enumitem}
\usepackage{tikz-cd}
\usetikzlibrary{decorations.pathreplacing}
\usepackage{amsfonts}
\usepackage{amsxtra}
\usepackage{euscript,mathrsfs}
\usepackage{color}
\usepackage{verbatim}
\usepackage[left=3.25cm,right=3.25cm,top=3cm,bottom=2cm]{geometry}
\usepackage[colorlinks=true, linktocpage=true, linkcolor=red!70!black, 
urlcolor  = blue!60!black,
citecolor=green!50!black]{hyperref}
\allowdisplaybreaks
\usepackage{algorithm}
\usepackage{algpseudocode}
\algrenewcommand\algorithmiccomment[1]{\hfill\textcolor{gray}{\textit{\# #1}}}

\usepackage{tikz-cd}
\usetikzlibrary{decorations.pathreplacing}
\usetikzlibrary{backgrounds}

\usepackage{enumitem}
\setenumerate{label={\rm (\alph{*})}}

\usepackage{amsfonts}
\usepackage{amsxtra}

\numberwithin{equation}{section}

\usepackage{xcolor}
\usepackage{graphicx}

\newcommand{\vertiii}[1]{{\left\vert\kern-0.25ex\left\vert\kern-0.25ex\left\vert #1 
		\right\vert\kern-0.25ex\right\vert\kern-0.25ex\right\vert}}

\newcommand{\R}{\mathbb R}

\renewcommand\ae{{a.\@e.\@}}

\newcommand\qq{\qquad}

\newcommand{\dif}{\mathrm{d}}

\newcommand{\bu}{\mathbf u}

\renewcommand{\dif}{\operatorname{d}\!}
\newcommand{\lebe}{\operatorname{L}}
\newcommand{\sobo}{\operatorname{W}}

\newcommand{\locc}{\operatorname{loc}}

\newcommand{\hold}{\operatorname{C}}

\newcommand{\sg}{\varepsilon}

\newcommand\norm[1]{\left\lVert#1\right\rVert}
\newcommand{\bv}{\mathbf{v}}

\newcommand{\ball}{\operatorname{B}}

\newcommand{\mres}{\mathbin{\vrule height 1.6ex depth 0pt width
		0.13ex\vrule height 0.13ex depth 0pt width 1.3ex}}

\allowdisplaybreaks

\newcommand{\setR}{\mathbb{R}}

\renewcommand\qq\qquad

\theoremstyle{plain}

\newtheorem*{theorem*}{Theorem}
\newtheorem{theorem}{Theorem}[section]
\newtheorem{lemma}[theorem]{Lemma}
\newtheorem{proposition}[theorem]{Proposition}
\newtheorem{corollary}[theorem]{Corollary}

\newtheorem{assumption}[theorem]{Assumption}

\theoremstyle{remark}

\newtheorem{remark}[theorem]{Remark}
\newtheorem{example}[theorem]{Example}

\begin{document}
	\title[Korn inequalities with trace and capacity terms]{On Korn inequalities \\ with lower order trace terms} 
	\author[F. Gmeineder]{Franz Gmeineder}
    \address{F.G.: Department of Mathematics and Statistics, Universit\"{a}t Konstanz, Universit\"{a}tsstrasse 10, 78464 Konstanz, Germany}
\email{franz.gmeineder@uni-konstanz.de}	
    
    \author[E. S\"{u}li]{Endre S\"{u}li}
     \address{E.S.: Mathematical Institute, University of Oxford, Radcliffe Observatory Quarter, OX2 6HG, Oxford, United Kingdom.}
\email{suli@maths.ox.ac.uk}	
	\author[T. Tscherpel]{Tabea Tscherpel}
     \address{T.T.: Department of Mathematics, Technische Universit\"{a}t Darmstadt, Dolivostra\ss e 15, 64293 Darmstadt, Germany}
\email{tscherpel@mathematik.tu-darmstadt.de}	


\date{\today}


	\maketitle
	\begin{abstract}
	We give an elementary estimate that entails and generalises numerous Korn inequalities scattered in the literature. As special instances, we obtain general Korn-type inequalities involving normal or tangential trace components, or lower dimensional trace integrals. 
	\end{abstract}
	\section{Introduction}
	Various problems from continuum and fluid mechanics require inequalities of \emph{Korn type}. 
    Such inequalities date back to {Korn}  \cite{Korn} (see also  {Friedrichs}  \cite{Friedrichs}) and, in their simplest form, allow to control full Sobolev norms by symmetric gradient Sobolev norms: If $1<p<\infty$ and $\Omega\subset\R^{n}$ is an open and bounded set with Lipschitz boundary $\partial\Omega$, then there exists a constant $c=c(n,p,\Omega)>0$ such that 
	\begin{align}\label{eq:Kornstart}
		\|\bu\|_{\sobo^{1,p}(\Omega)}\leq c\big(\|\bu\|_{\lebe^{p}(\Omega)} + \|\sg(\bu)\|_{\lebe^{p}(\Omega)} \big)\qquad\text{for all}\;\bu\in\sobo^{1,p}(\Omega;\R^{n}). 
	\end{align}
	Here, $\sg(\bu):=\frac{1}{2}(D\bu + (D\bu)^{\top})$ denotes the symmetric gradient, i.e., the symmetric part of the full gradient $D\bu$. Inequalities of this kind play an important role in various well-posedness results for partial differential equations and variational problems, and have been refined in various directions; see, e.g.,  \cite{BreitCianchiDiening,BreitDiening,Cianchi,Ciarlet1,Ciarlet2,Dain,DuvautLions,Fuchs1994,Gobert,Kalamajska,Necas,Smith} for an incomplete list, also including  applications in continuum and fluid mechanics or general relativity.  
	
	In numerous applications, weaker versions than inequalities~\eqref{eq:Kornstart} are required. For instance, if weak formulations of partial differential equations only ensure control of $\|\sg(\bu)\|_{\lebe^{p}(\Omega)}$ and come with Dirichlet, normal or tangential conditions on $\bu$ on a part $\Gamma\subset\partial\Omega$ of the boundary, a natural substitute of \eqref{eq:Kornstart} is an estimate of the form 
    \begin{align}\label{eq:Kornstart1}
		\|\bu\|_{\sobo^{1,p}(\Omega)}\leq c\big(\|T(\bu)\|_{\lebe^{p}(\Gamma)} + \|\varepsilon(\bu)\|_{\lebe^{p}(\Omega)}\big)\quad\text{for all}\;\bu\in\sobo^{1,p}(\Omega;\R^{n}). 
	\end{align}
	Dirichlet or partial trace conditions (such as hypotheses on the normal or tangential traces) then are reflected by the choices  $T(\bu)=\mathrm{tr}_{\partial\Omega}(\bu)$ with the boundary trace operator $\mathrm{tr}_{\partial\Omega}$, $T(\bu)=\mathrm{tr}_{\partial\Omega}(\bu)\cdot\nu_{\partial\Omega}$ or $T(\bu)=\mathrm{tr}_{\partial\Omega}(\bu)-(\mathrm{tr}_{\partial\Omega}(\bu)\cdot\nu_{\partial\Omega})\nu_{\partial\Omega}$; here $\nu_{\partial\Omega}$ denotes the outer unit normal to $\partial\Omega$. Variants of such estimates have been obtained by {Ryzhak} \cite{Ryzhak} and  {Desvillettes \& Villani} \cite{DesvilletesVillani} in the homogeneous case $T(\mathbf{u})=0$. 
{Versions of \eqref{eq:Kornstart1} with $\Gamma \subsetneq \partial \Omega$ naturally occur in the study of physical problems if periodic or natural boundary conditions are imposed on the remaining boundary $\partial \Omega \setminus \Gamma$.} We point out that homogeneous variants of \eqref{eq:KornStart}  are not sufficient to yield inequalities of the form \eqref{eq:Kornstart1}; see Remark \ref{rem:hominhom} below.

	There is also a number of models, however, that requires to control weaker expressions than $\sg(\bu)$. For instance, models that take into account torsions but no volumetric changes are  usually based on the trace-free gradient $\nabla^{D}\mathbf{u}\coloneqq \nabla\mathbf{u}-\frac{1}{n}\mathrm{div}(\mathbf{u})\mathbbm{1}_{n}$ or the trace-free symmetric gradient $\sg^{D}(\bu)=\sg(\bu)-\frac{1}{n}\mathrm{div}(\bu)\mathbbm{1}_{n}$ with the $(n\times n)$-unit matrix $\mathbbm{1}_{n}$. 
    Perhaps surprisingly, much less is known in this case; only in a more  recent contribution of {Bauer \& Pauly}  \cite{BauerPauly}, corresponding variants of \eqref{eq:Kornstart1} \emph{for homogeneous boundary conditions} have been addressed. { On the other hand, inhomogeneous cases are particularly important in the numerical analysis of PDEs. In particular, this concerns methods which approximate the respective PDEs and  treat the boundary conditions by means of penalisation \cite{GazcaEtAl,KaltenbachWichmann,Verfuerth}.} 

	In all cases, the available strategies of proof strongly hinge on the specific structure of the operators $\sg$ or $\sg^{D}$ \emph{in conjunction} with the specific boundary conditions. 
    This, in turn, motivates the study of inequalities \eqref{eq:Kornstart1} for a more flexible, yet natural class of differential operators $\mathbb{A}$ and, in some sense, the optimal generalisations of the terms $\|T(\bu)\|_{\lebe^{p}(\Gamma)}$ as in \eqref{eq:Kornstart1}. The present paper is precisely concerned with such generalisations.

More precisely, we assume that $\mathbb{A}$ is a linear, $k$-th order, homogeneous constant coefficient differential operator on $\R^{n}$ between the finite dimensional inner product spaces $V$ and $W$. 
This means that $\mathbb{A}$ has a representation
	\begin{align}\label{eq:form}
		\mathbb{A}=\sum_{|\alpha|=k}\mathbb{A}_{\alpha}\partial^{\alpha}\;\;\;\text{with corresponding Fourier symbol}\;\;\;\mathbb{A}[\xi] = \sum_{|\alpha|=k}\xi^{\alpha}\mathbb{A}_{\alpha}.
	\end{align}
	Here, for each {$\alpha \in \mathbb{N}_0^n$ with} $|\alpha|=k$, $\mathbb{A}_{\alpha}\colon V\to W$ is a fixed linear map. To state our main result, we require two more notions: We say that $\mathbb{A}$ given by \eqref{eq:form} is \emph{elliptic} if, for any $\xi\in\R^{n}\setminus\{0\}$, $\mathbb{A}[\xi]\colon V\to W$ is injective. If, for any $\xi\in\mathbb{C}^{n}\setminus\{0\}$, $\mathbb{A}[\xi]\colon V+\mathrm{i}V\to W+\mathrm{i}W$ remains injective, then $\mathbb{A}$ is called \emph{$\mathbb{C}$-elliptic} \cite{BreitDieningGmeineder,Kalamajska,Smith}. This framework includes, for instance, the (trace-free) symmetric gradients, see Example \ref{ex:diffops}  below. 
    Deferring the formal definition of $\sobo^{k}X$, which the reader might understand as a generalisation of $\sobo^{k,p}=\sobo^{k}\lebe^{p}$ for $X=\lebe^{p}$, we have: 
    
		\begin{theorem}\label{thm:korn-gen}
		Let $\Omega\subset \setR^n$ be open and bounded, and let $\mathbb{A}$ be a $k$-th order $\mathbb{C}$-elliptic differential operator of the form \eqref{eq:form}. Moreover, let $(X(\Omega),\|\cdot\|_{X(\Omega)})$ and let $(Y(\Omega),\|\cdot\|_{Y(\Omega)})$ be two infinite-dimensional Banach function spaces with $X(\Omega),Y(\Omega)\subset\lebe_{\locc}^{1}(\Omega;V)$. If 
	\begin{enumerate}[label=\emph{(\roman*)}]
		\item\label{item:Korn1A} the inequalities 
		\begin{align}\label{eq:kornetto0}
			\begin{split}
				\norm{\bu}_{\sobo^{k}Y(\Omega)} &\lesssim \norm{\bu}_{X(\Omega)} + \norm{\mathbb{A}\bu}_{X(\Omega)} \\
				\inf_{\bm{\rho}\in\ker(\mathbb{A};\R^{n})}\norm{\bu-\bm{\rho}}_{X(\Omega)} & \lesssim \|\mathbb{A}\bu\|_{X(\Omega)},
			\end{split}
		\end{align}
		hold for all $\bu\in\sobo^{k}X(\Omega;V)$, where
		\begin{align}\label{eq:nullspacemaintheorem}
			\ker(\mathbb{A};\R^{n}):= \{\bv\in\mathscr{D}'(\R^{n};V)\colon\;\mathbb{A}\bv=0\} = \{\mathbf{v}\colon\R^{n}\to V\;\text{polynomial}\colon\;\mathbb{A}\mathbf{v}=0\},
		\end{align}
		\end{enumerate}
        then 
        \begin{enumerate}[label=\emph{(i\roman*)}]
        \item\label{item:Korn1B} the following are equivalent for a seminorm $\vertiii{\cdot}$ on $\sobo^{k}X(\Omega;V)$: 
		\begin{enumerate}
			\item\label{item:Korn1} $\vertiii{\cdot}$ is a norm on $\ker(\mathbb{A};\R^{n})|_{\Omega}$. 
			\item\label{item:Korn2} We have validity of the \emph{Korn-type inequality}
			\begin{align}\label{eq:improvedKorn}
				\|\bu\|_{\sobo^{k}Y(\Omega)} \lesssim \vertiii{\bu} + \|\mathbb{A}\bu\|_{X(\Omega)}\qquad\text{for all}\;\bu\in\sobo^{k}X(\Omega). 
			\end{align}
			\item\label{item:Korn3} We have validity of the \emph{Korn-type inequality}
			\begin{align}\label{eq:improvedKorn1}
				\|\bu\|_{\sobo^{k}Y(\Omega)} \lesssim  \|\mathbb{A}\bu\|_{X(\Omega)}\qquad\text{for all}\;\bu\in\sobo^{k}X(\Omega)\;\text{with}\;\vertiii{\bu}=0. 
			\end{align}
		\end{enumerate}
		\end{enumerate} 
	\end{theorem}
    Theorem \ref{thm:korn-gen} comes with a natural variant, see Theorem \ref{thm:variant}. 
	This lets us retrieve several known Korn-type inequalities, e.g.,  \eqref{eq:Kornstart1} for several choices of $\Gamma$ and $\vertiii{\mathbf{u}}=\|T(\mathbf{u})\|_{\lebe^{p}(\Gamma)}$, but also opens the door to novel inequalities. For future reference, note that $\eqref{eq:kornetto0}_{1}$ is referred to as the \emph{first Korn inequality} or \emph{Korn inequality of the first kind}.
    
    The proofs of Theorem \ref{thm:korn-gen} and a variant thereof as given in Theorem \ref{thm:variant} below are by no means difficult, yet they give a unifying framework for a variety of coercive inequalities required in applications. There are several ways to establish Theorem \ref{thm:korn-gen}, and Section \ref{sec:proof} provides two proofs: one direct and one indirect one. 
 The first option is preferable, since it allows -- at least in principle -- to track the underlying  constants. The latter is essential, e.g., if the inequalities are applied in the context of numerical approximations and the original domains are approximated by a sequence of polyhedral ones. 
    While the right-hand side of the estimate is typically controlled by a priori estimates, this may provide uniform bounds on the sequence of approximate solutions. This, in turn, is key to establish  convergence results based on compactness arguments.
    
    The direct argument given in Section \ref{sec:directproof} can be used  to yield quantitative bounds on the underlying constants; see also Remark \ref{rem:explicitconstants}. Yet, we do not aim for optimal constants here. In this respect, the present paper differs, e.g., from \cite{BauerPauly,DesvilletesVillani}, where a considerable focus is put on the underlying sharp constants. Moreover, allowing for $X\neq Y$ in Theorem \ref{thm:korn-gen} is a natural requirement, e.g., when considering models that account for logarithmic hardening and are thus based on certain Orlicz spaces; see Section \ref{sec:otherscales}.

    The key direction '\emph{\ref{item:Korn1B}}\ref{item:Korn1}$\Rightarrow$\emph{\ref{item:Korn1B}}\ref{item:Korn2}' from Theorem \ref{thm:korn-gen} can equally be inferred from an abstract  result due to {Dra\v{z}i\'{c}} \cite{Drazic,Drazic1} (see also  Jovanovi\'{c} \& S\"{u}li \cite[Thm.~1.9]{JovanovicSuli}), in turn relying on a simple indirect argument to be revisited in Section \ref{sec:drazic} below. {Dra\v{z}i\'{c}}'s approach does not allow to track constants, but is of independent interest and seems to have gone widely unnoticed in the realm of Korn-type inequalities. However, Dra\v{z}i\'{c}'s Lemma \ref{lem:drazic} has natural limitations, and cannot be applied in situations where certain compact embeddings are not available. This especially pertains to scenarios where one aims to go beyond Korn inequalities. An instance thereof are inequalities on $\lebe^{1}$-based spaces: It is here where a trivial modification of the direct argument from Section \ref{sec:directproof} still applies, whereas the indirect argument from Section \ref{sec:drazic} does not; see Proposition \ref{prop:1} and the subsequent discussion for more details.

    Sections \ref{sec:choices0} and \ref{sec:choices} deal with particularly important special cases of Theorem \ref{thm:korn-gen} for several constellations of domains $\Omega$, differential operators $\mathbb{A}$, seminorms $\vertiii{\cdot}$ and function spaces $X,Y$. This may serve as a toolbox for various applications, and the key results concerning operators arising frequently in applications are summarised in Table \ref{table:InequalitiesSummary}. More precisely, Section \ref{sec:choices0} deals with seminorms defined in terms of Lebesgue integrals with respect to $\mathscr{L}^{n}$ or (partial) boundary trace conditions such as tangential or normal traces. Similar to \cite{BauerPauly1,DesvilletesVillani}, the validity of the corresponding Korn-type inequalities strongly depends on the interplay between the regularity and the geometry of the domains, the differential operators and the chosen seminorms. A similar principle applies to the situation where the seminorms are given by (interior) trace integrals with respect to lower dimensional measures $\mu$. For the latter to be meaningfully defined on Sobolev spaces, it is clear that they must meet certain capacitary criteria. Subject to such natural assumptions, the validity of the corresponding Korn-type inequalities is addressed in detail in Section \ref{sec:choices}. 
    
    In general, depending on the domains, operators and seminorms, it is difficult to \emph{classify} all potential constellations that lead to the corresponding Korn-type inequalities. In particular, this partially necessitates a switch between purely algebraic and analytic problems, see Example \ref{ex:devsymgradcomplications}. For such scenarios and in view of applications, Section \ref{sec:sym-comput} provides an algorithm including implementation to verify whether certain conditions required for the validity of Korn-type inequalities are met. 

    Lastly, we focus on inequalities which bound full Sobolev norms. It would be equally interesting to discuss this theme in view of incompatible scenarios \`a la Neff et al. \cite{GmeinederLewintanNeff1,GmeinederLewintanNeff2,Neff0,Neff00,Neff1} or nonlinear variants \`a la Ciarlet et al. \cite{Ciarlet1,Ciarlet3},  topics which we do not touch here.
\section*{Notation}
We briefly comment on notation and background terminology. Throughout, $\Omega\subset\R^{n}$ denotes an open set. We write $\mathscr{M}(\Omega)$ or $\mathscr{M}_{\geq 0}(\Omega)$ for the measurable functions with values in $[-\infty,\infty]$ or $[0,\infty]$, respectively. Even though we will not employ \ref{itm:BNF-1}--\ref{itm:BNF-5} from below explicitly, we collect the definition of Banach function spaces as appearing in Theorem \ref{thm:korn-gen} for the sake of completeness. Following Bennett and Sharpley \cite{BennettSharpley}, we call a map  $\rho\colon\mathscr{M}_{\geq 0}(\Omega)\to[0,\infty]$ a \emph{Banach function norm} if it satisfies the following for all $u,v,u_{1},u_{2},...\in\mathscr{M}_{\geq 0}(\Omega)$ and $\lambda\geq 0$:  
\begin{enumerate}[label = (B\arabic*)]
\item \label{itm:BNF-1} 
If $\rho(u)=0$, then $u=0$ $\mathscr{L}^{n}$-a.e.. Moreover,   $\rho(\lambda u)=\lambda \rho(u)$ and $\rho(u+v)\leq \rho(u)+\rho(v)$.
\item  \label{itm:BNF-2} 
If $u\leq v$ $\mathscr{L}^{n}$-a.e., then $\rho(u)\leq \rho(v)$. 
\item \label{itm:BNF-3} 
If $u_{i}\nearrow u$ $\mathscr{L}^{n}$\text{-a.e.~in $\Omega$}, then $\rho(u_{i})\nearrow \rho(u)$. 
\item \label{itm:BNF-4} 
If $A\subset\Omega$ is Lebesgue measurable and satisfies $\mu(A)<\infty$, then $\rho(\mathbbm{1}_{A})<\infty$. 
\item \label{itm:BNF-5} 
If $A\subset\Omega$ is Lebesgue measurable with $\mathscr{L}^{n}(A)<\infty$, then there exists $c=c(A,\rho)$ such that $\|u\|_{\lebe^{1}(A)}\leq c\rho(u)$. 
\end{enumerate}
For a Banach function norm $\rho$, the associated \emph{Banach function space} is given by 
\begin{align*}
X_{\rho}(\Omega)\coloneqq \{u\in\mathscr{M}(\Omega)\colon\;\|u\|_{X_{\rho}(\Omega)}\coloneqq \rho(|u|)<\infty\},
\end{align*}
and we also abbreviate  $X(\Omega)\coloneqq X_{\rho}(\Omega)$. For $k\in\mathbb{N}$, the associated Sobolev space is defined  by 
\begin{align*}
\sobo^{k}X(\Omega)\coloneq\Big\{u\colon\Omega\to\R\;\text{$k$-times weakly differentiable}\colon\;\|u\|_{\sobo^{k}X(\Omega)}\coloneqq\sum_{|\alpha|\leq k}\|\partial^{\alpha}u\|_{X_{\rho}(\Omega)}<\infty\Big\}. 
\end{align*}
For a finite-dimensional vector space $V$, the space $\sobo^{k}X(\Omega;V)$ of $V$-valued $\sobo^{k}X$-maps is defined analogously. As usual, $\mathscr{F}$ denotes the Fourier transform on the Schwartz functions or tempered distributions, respectively. We write $\mathcal{M}^{+}(\Omega)$ for the positive, finite Borel measures on $\Omega$, and $\mathcal{M}_{c}^{+}(\Omega)$ for the elements of $\mathcal{M}^{+}(\Omega)$ with compact support inside $\Omega$. Lastly, we write $A\lesssim B$ to express $A\leq cB$ with a constant $c>0$ independent of $B$.

\section{Proof of Theorem \ref{thm:korn-gen} and related inequalities}\label{sec:proof}

\subsection{Differential operators and examples} Throughout the paper, let $n\geq 2$. We start by collecting some background material on the differential operators from \eqref{eq:form}; see \cite{BreitDieningGmeineder,Kalamajska,Smith}. The importance of $\mathbb{C}$-ellipticity lies in the following result. 

\begin{lemma}\label{lem:Celliptic} For a differential operator $\mathbb{A}$ of the form \eqref{eq:form}, the following are equivalent:
\begin{enumerate}
	\item\label{item:Cell1} $\mathbb{A}$ is $\mathbb{C}$-elliptic. 
	\item\label{item:Cell2} For any open, bounded and connected $\Omega\subset\R^{n}$, $\ker(\mathbb{A};\Omega)\coloneqq\{\mathbf{u}\in\mathscr{D}'(\Omega;V)\colon\;\mathbb{A}\mathbf{u}=0\}$ is a finite-dimensional subspace of the $V$-valued polynomials on $\R^{n}$. 
    \item There exists an open and connected subset $\Omega\subset\R^{n}$ such that $\ker(\mathbb{A};\Omega)$ is a finite-dimensional subspace of the $V$-valued polynomials. 
\end{enumerate}
\end{lemma}
For future reference, we illustrate the previous lemma with an example. This shall be continued in Sections \ref{sec:choices0} and \ref{sec:choices}.
\begin{example}[$\mathbb{C}$-elliptic operators and their null  spaces]\label{ex:diffops} Throughout, we consider maps $\mathbf{u}\colon\R^{n}\to\R^{n}$ with $n\geq 2$. 
\begin{enumerate}
	\item\label{item:devgrad} \emph{Deviatoric gradient.} The operator $\nabla^{D}\mathbf{u}\coloneqq \mathbf{u} - \frac{1}{n}\mathrm{div}(\mathbf{u})\mathbbm{1}_{n}$ is $\mathbb{C}$-elliptic. Since this is lesser known, we briefly show this by computing its null space. 
    For $\mathbf{u}=(u_{1},...,u_{n})$, $\nabla^{D}\mathbf{u}=0$ means that $\partial_{j}u_{k}=0$ whenever $j\neq k$ and 
\begin{align}\tag{$I_{j}$}
	\frac{n-1}{n}\partial_{j}u_{j} = \frac{1}{n}\sum_{i\neq j}\partial_{i}u_{i}\qquad\text{for all}\;j\in\{1,...,n\}.
	\end{align}
	Subtracting $(I_{j})$ from $(I_{k})$ we obtain $\partial_{j}u_{j}=\partial_{k}u_{k}$ for all $j,k\in\{1,...,n\}$. Fix $j\in\{1,...,n\}$. Then $\partial_{k}u_{j}=0$ for $k\neq j$ yields that $u_{j}(x)=f_{j}(x_{j})$ and $\partial_{j}f_{j}(x_{j})=\partial_{k}f_{k}(x_{k})$. This implies that there exists $\alpha\in\R$ such that $\partial_{j}f_{j}(x_{j})=\alpha$. Hence $f_{j}(t)=\alpha t + b_j$. Therefore we have  
	$\mathbf{u}(x)= \alpha x + {b}$, and thus the kernel is given by 
    \begin{align}\label{def:ker-devgrad}
        \ker(\nabla^D) = \{x \mapsto \alpha x + {b} \colon \alpha \in \mathbb{R}, {b} \in \mathbb{R}^n\}. 
    \end{align}
	\item\label{item:symgrad} \emph{Symmetric gradient.} Following Reshetnyak \cite{Reshetnyak}, the null space of $\varepsilon$ on $\R^{n}$ is given by the rigid deformations 
	\begin{align}\label{def:ker:symgrad}
	 \ker(\varepsilon) = \mathscr{R}(\R^{n}) = \{x \mapsto Ax+{b}\colon\;A\in\R_{\mathrm{skew}}^{n\times n},\, {b}\in\R^{n}\}. 
	\end{align}
	\item  \label{item:devsymgrad}
    \emph{Deviatoric symmetric gradient.} Following Reshetnyak \cite{Reshetnyak}, provided that $n\geq 3$, the null space of $\varepsilon^{D}$ is given by the \emph{conformal Killing fields}
    \begin{equation}\label{def:ker-devsymgrad}
	\begin{aligned}
	\ker(\varepsilon^D) = \mathscr{K}(\R^{n})\coloneqq \big\{ x \mapsto 2 \left\langle {a},x \right\rangle x &- |x|^2 {a} +  Ax + \alpha x + {b} \colon \\ &  A\in\mathbb{R}_{\mathrm{skew}}^{n\times n},\; {a},{b}\in\R^{n},\; \alpha \in \R\big\},
	\end{aligned}
    \end{equation}
    see also Dain~\cite{Dain}. 
    As discussed, e.g., in \cite[Ex.~2.2]{BreitDieningGmeineder}, the null space of $\varepsilon^{D}$ contains a copy of the holomorphic functions on $\mathbb{C}$. Hence, for $n=2$, the null space of  $\varepsilon^{D}$ is not finite-dimensional and thus $\mathbb{C}$ is not $\mathbb{C}$-elliptic. However, in the terminology given prior to Theorem \ref{thm:korn-gen}, $\sg^{D}$ is elliptic in two dimensions. 
\end{enumerate}
\end{example}

\subsection{A direct proof of Theorem \ref{thm:korn-gen}}\label{sec:directproof}
	We now come to the proof of the main result of the present paper; we shall first give a self-contained proof. 

	\begin{proof}[Proof of Theorem \ref{thm:korn-gen}]
Suppose that \emph{\ref{item:Korn1A}} holds. On '\ref{item:Korn1}$\Rightarrow$\ref{item:Korn2}'. By Lemma \ref{lem:Celliptic}, the null space $\ker(\mathbb{A};\R^{n})$ as in \eqref{eq:nullspacemaintheorem} is a finite-dimensional space of $V$-valued polynomials.  Hence, for each $\bu\in\sobo^{k}X(\Omega;V)$, the infimum on the left-hand side of   $\eqref{eq:kornetto0}_{2}$ is attained for some $\bm{\rho}_{0}\in\ker(\mathbb{A};\R^{n})$. By slight abuse of notation, we write $\bm{\rho}_{0}$ for its restriction to $\Omega$, whereby $\bm{\rho}_{0}\in\sobo^{k}X(\Omega;V)$. 
       Applying~\eqref{eq:kornetto0} yields 
		\begin{align}
			\|\bu\|_{\sobo^{k}Y(\Omega)} & \leq \|\bu-\bm{\rho}_{0}\|_{\sobo^{k}Y(\Omega)} + \|\bm{\rho}_{0}\|_{\sobo^{k}Y(\Omega)} \notag\\ 
			& \!\!\!\!\stackrel{\eqref{eq:kornetto0}_{1}}{\lesssim} \|\bu-\bm{\rho}_{0}\|_{X(\Omega)} +  \|\mathbb{A}\bu\|_{X(\Omega)} + \|\bm{\rho}_{0}\|_{\sobo^{k}Y(\Omega)} \label{eq:intermediate1}\\ 
			& \!\!\!\!\stackrel{\eqref{eq:kornetto0}_{2}}{\lesssim}  \|\mathbb{A}\bu\|_{X(\Omega)} + \|\bm{\rho}_{0}\|_{\sobo^{k}Y(\Omega)} \eqqcolon A + B. \notag
		\end{align}
		The term $A$ is already in the desired form. For $B$, we note that all norms on finite-dimensional spaces are equivalent. Since $\vertiii{\cdot}$ is a norm on $\ker(\mathbb{A};\R^{n})|_{\Omega}$ by \ref{item:Korn1}, this particularly entails that 
		\begin{align}\label{eq:intermediate2}
			\|\pi\|_{\sobo^{k}Y(\Omega)} \lesssim \vertiii{\pi}\qquad\text{for all}\;\pi\in\ker(\mathbb{A};\R^{n})|_{\Omega}. 
		\end{align}
        Let $\Pi\colon\sobo^{k}X(\Omega;V)\to\ker(\mathbb{A};\R^{n})|_{\Omega}$ be an arbitrary but fixed projection onto $\ker(\mathbb{A};\R^{n})|_{\Omega}$. Then we have with \eqref{eq:intermediate2} that
		\begin{align*}
			B & = \|\bm{\rho}_{0}\|_{\sobo^{k}Y(\Omega)} = \|\Pi(\bm{\rho}_{0})\|_{\sobo^{k}Y(\Omega)} \\ &  \leq  \|\Pi(\bu-\bm{\rho}_{0})\|_{\sobo^{k}Y(\Omega)} + \|\Pi(\bu)\|_{\sobo^{k}Y(\Omega)}\\
            & \lesssim  \|\Pi(\bu-\bm{\rho}_{0})\|_{\sobo^{k}Y(\Omega)} + \vertiii{\Pi(\bu)}.
		\end{align*}
        Because of $\dim(\ker(\mathbb{A};\R^{n})|_{\Omega})<\infty$, the operator $\Pi\colon\sobo^{k}X(\Omega)\to\ker(\mathbb{A};\R^{n})|_{\Omega}$ is a bounded linear operator regardless of the norm with which $\ker(\mathbb{A};\Omega)$ is endowed. In conclusion, 
        \begin{align*}
			B  \lesssim \|\bu-\bm{\rho}_{0}\|_{\sobo^{k}X(\Omega)} + \vertiii{\bu} &  \stackrel{\eqref{eq:kornetto0}_{2}}{\lesssim}  \|\mathbb{A}\bu\|_{X(\Omega)} + \vertiii{\bu},
		\end{align*}
	and hence~\eqref{eq:improvedKorn}  follows	in view of \eqref{eq:intermediate1}.
		
		The direction `\ref{item:Korn2}$\Rightarrow$\ref{item:Korn3}' is immediate. On '\ref{item:Korn3}$\Rightarrow$\ref{item:Korn1}'. Let $\bu\in\ker(\mathbb{A};\Omega)$ satisfy $\vertiii{\bu}=0$. Then \eqref{eq:improvedKorn1} gives us $\|\bu\|_{\sobo^{k}X(\Omega)}=0$, whereby $\bu=0$. Since $\vertiii{\cdot}$ is a priori assumed to be a seminorm on $\sobo^{k}X(\Omega;V)$, this implies \ref{item:Korn1}, and the proof is complete. 
	\end{proof}
	
	\begin{remark}[On the assumptions of Theorem \ref{thm:korn-gen}]
		Even though it might seem at a first glance that Theorem \ref{thm:korn-gen} holds  without any conditions on the boundary $\partial\Omega$, this is not so: Indeed,  the validity of \eqref{eq:kornetto0} rules out certain domains. For instance, if $X=\lebe^{p}$ with $1<p<\infty$ and $\mathbb{A}=\varepsilon$ is the symmetric gradient, an almost optimal condition on $\Omega$ for $\eqref{eq:kornetto0}_{1}$ to hold is to be John, see \cite{Kauranen}. More generally, and as will be explained in Section \ref{sec:validityfirst} below, inequalities \eqref{eq:kornetto0} hold true for all John domains and all reasonable choices of $X$, see \cite{DieningGmeineder,DieningRuzickaSchumacher}.
	\end{remark}
	\begin{remark}[Choices of $X$ and $Y$]
		Inequalities \eqref{eq:kornetto0}, which represent the main hypotheses for Theorem \ref{thm:korn-gen}, are satisfied for basically all reasonable choices of function spaces $X,Y$ and domains $\Omega\subset\R^{n}$. In the framework of the classical Korn inequalities, we have $X=Y=\lebe^{p}$ for some $1<p<\infty$. However, as explained in detail by Cianchi \cite{Cianchi} in the framework of Orlicz spaces $\lebe^{\varphi}$, it is often useful to include different choices of $X$ and $Y$. For instance, Korn inequalities involving Young functions $\varphi$ of $\lebe\log^{1+\alpha}\lebe$-type usually come with the loss of one logarithm; see Section~\ref{sec:otherscales}  for more detail.
	\end{remark}

\subsection{{Dra\v{z}i\'{c}'s lemma and an indirect proof of Theorem \ref{thm:korn-gen}}}\label{sec:drazic}
	The following lemma is a special case of a result due to Dra\v{z}i\'{c} \cite{Drazic,Drazic1} and is established as an abstract variant of Lions' or Ne\v{c}as' classical approach to coercive inequalities \cite{DuvautLions,Necas}. 
    
	\begin{lemma}[Dra\v{z}i\'{c}]\label{lem:drazic}
	Let $(\mathbb{X}_{0},\|\cdot\|_{0})$, $(\mathbb{X}_{1}, \|\cdot\|_{1})$, for $i \in \{1,2\}$, be two Banach spaces such that we have the compact embedding  $\mathbb{X}_{1}\hookrightarrow\hookrightarrow\mathbb{X}_{0}$. Moreover, let $S_{i}\colon\mathbb{X}_{i}\to\R_{\geq 0}$ be two bounded sublinear functionals such that there exists a constant $c>0$ with 
	\begin{align}\label{eq:KornStart}
	\|x\|_{1}\leq c\,(S_{0}(x) + S_{1}(x))\qquad\text{for all}\;x\in\mathbb{X}_{1}. 
	\end{align}
Then there exists a constant $C>0$ such that 
\begin{align}\label{eq:Drazic1}
\inf_{z\in\ker(S_{1})}\|x-z\|_{1}\leq C S_{1}(x)\qquad\text{for all}\;x\in\mathbb{X}_{1}. 
\end{align}
In particular, if $\ker(S_{1})=\{0\}$, then there exists a constant $C>0$ such that 
		\begin{align}\label{eq:KornConclude}
		\|x\|_{1}\leq C\,S_{1}(x)\qquad\text{for all}\;x\in\mathbb{X}_{1}. 
	\end{align}
	\end{lemma}
 Inequality \eqref{eq:KornConclude} shall be most important to us. For the reader's convenience, we thus pause to give a self-contained proof.
 \begin{proof}[Proof of \eqref{eq:KornConclude}]
	Suppose that \eqref{eq:KornConclude} does not hold. Then there exists a sequence $(x_{j})\subset\mathbb{X}_{1}$ such that $\|x_{j}\|_{1}=1$ for all $j\in\mathbb{N}$ and $S_{1}(x_{j})\to 0$ as $j\to\infty$. Because of $\mathbb{X}_{1}\hookrightarrow\hookrightarrow\mathbb{X}_{0}$, there exists a (non-relabelled) subsequence and some $x\in \mathbb{X}_{0}$ such that $x_{j}\to x$ strongly in $\mathbb{X}_{0}$. 
    Since $(x_{j})$ is Cauchy in $\mathbb{X}_{0}$, by \eqref{eq:KornStart} we have: 
	\begin{align*}
\limsup_{j,l\to\infty}	\|x_{j}-x_{l}\|_{1} & \leq c\,\limsup_{j,l\to\infty}(\|S_{0}\|\,\|x_{j}-x_{l}\| + S_{1}(x_{j})+S_{1}(x_{l})) \\ & \leq c\|S_{0}\|\limsup_{j,l\to\infty} \|x_{j}-x_{l}\|_{0} = 0. 
	\end{align*}
	Hence, $(x_{j})$ is Cauchy in $(\mathbb{X}_{1},\|\cdot\|_{1})$ and, by the Banach space property of $\mathbb{X}_{1}$, converges to some $\widetilde{x}\in\mathbb{X}_{1}$; therefore, one has $x=\widetilde{x}$. Since $\|x_{j}\|_{1}=1$ for all $j\in\mathbb{N}$, it follows that $\|\widetilde{x}\|_{1}=1$. On the other hand, 
	\begin{align*}
|S_{1}(x)| =	|S_{1}(\widetilde{x})|\leq |S_{1}(\widetilde{x}-x_{j})| + |S_{1}(x_{j})| \leq \|S_{1}\|\,\|\widetilde{x}-x_{j}\|_{1} + |S_{1}(x_{j})| \to 0 
	\end{align*}
	as $j\to\infty$. In conclusion, $x\in\ker(S_{1})=\{0\}$, which is at variance with $\|x\|_{1}=1$. This completes the proof. 
	\end{proof}
Applying Lemma \ref{lem:drazic}, we can give an 
	\begin{proof}[Indirect proof of Theorem~\ref{thm:korn-gen}] We adopt the notation from Lemma \ref{lem:drazic}. In the framework of Lemma~\ref{lem:drazic}, we put $\mathbb{X}_{1}=\sobo^{k}Y(\Omega;V)$,  $\mathbb{X}_{0}\coloneqq\sobo^{k-1}X(\Omega;V)$,  
		\begin{align*}
		S_{0}(\mathbf{u})\coloneqq\|\mathbf{u}\|_{\mathbb{X}_{0}}\;\;\;\text{and}\;\;\;S_{1}(\mathbf{u})\coloneqq\vertiii{\mathbf{u}}+\|\mathbb{A}\mathbf{u}\|_{X(\Omega)}. 
		\end{align*}
		Clearly, if \eqref{eq:kornetto0} holds, then we have \eqref{eq:KornStart}.
	If $\vertiii{\cdot}$ is a norm on $\ker(\mathbb{A};\Omega)$, then $\ker(S_{1})=\{0\}$. In consequence, \eqref{eq:improvedKorn} follows from \eqref{eq:KornConclude}, completing the proof. 
	\end{proof}
For future reference, we discuss the necessity of the hypotheses underlying Dra\v{z}i\'{c}'s Lemma~\ref{lem:drazic}. By this, we shall find its natural limitations in view of applications. Apart from completeness, we note that the compact embedding hypothesis is essential and  cannot be omitted:
\begin{example}\label{ex:MazyaDrazic}
Let $\Omega\subset\R^{2}$ be Maz'ya's 'shrinking rooms and passages domain', see Fig.~\ref{fig:mazya}. Put 
\begin{align*}
\mathrm{V}^{2,p}(\Omega)\coloneqq\{u\in\lebe^{p}(\Omega)\colon\;\nabla^{2}u\in\lebe^{p}(\Omega;\R^{2\times 2})\}
\end{align*}
and endow it with its canonical norm $\|u\|_{\mathrm{V}^{2,p}(\Omega)}\coloneqq\|u\|_{\lebe^{p}(\Omega)}+\|\nabla^{2}u\|_{\lebe^{p}(\Omega)}$. We define $\mathbb{X}_{0} \coloneqq \sobo^{1,p}(\Omega)$ and $\mathbb{X}_{1}\coloneqq\sobo^{2,p}(\Omega)$, so that $\mathbb{X}_{1}\hookrightarrow\mathbb{X}_{0}$. 
{However, as a byproduct of the following we shall see that the embedding is not compact.}  We now choose  $S_{0}(u) \coloneqq \|u\|_{\sobo^{1,p}(\Omega)}$ and $S_{1}(u)\coloneqq \|u\|_{\mathrm{V}^{2,p}(\Omega)}$, so that $S_{i}\colon\mathbb{X}_{i}\to[0,\infty)$ are bounded norms for $i\in\{1,2\}$ {and $\ker(S_1) = \{0\}$}. 
Clearly, one has that 
\begin{align*}
\|u\|_{\sobo^{2,p}(\Omega)} \leq S_{0}(u)+S_{1}(u)\qquad\text{for all}\;u\in\sobo^{2,p}(\Omega),  
\end{align*}
which means that \eqref{eq:KornStart} is satisfied. 
However, one can show that there is no constant $C>0$ such that 
\begin{align}\label{eq:contra}
\|u\|_{\sobo^{2,p}(\Omega)}\leq CS_{1}(u)=C\|u\|_{\mathrm{V}^{2,p}(\Omega)}\qquad\text{for all}\;u\in\sobo^{2,p}(\Omega),
\end{align}
i.e., \eqref{eq:KornConclude} does not hold. 
Indeed, suppose that \eqref{eq:contra} holds. 
Based on the construction indicated in Figure \ref{fig:mazya}, for $N \in \mathbb{N}$ we define 
\begin{align*}
u_{N}(x_{1},x_{2})\coloneqq\sum_{j=1}^{N}\left(\mathbbm{1}_{P_{j}}(x_{1},x_{2})(2^{j+1}(x_{2}-1)-1) + \mathbbm{1}_{R_{j}}(x_{1},x_{2})(2^{j}(x_{2}-1)^{2})\right).
\end{align*}
As established by Maz'ya in \cite[Chpt.~1.1.4, Ex.~2]{Mazya}, $\sup_{N\in\mathbb{N}}\|u_{N}\|_{\mathrm{V}^{2,p}(\Omega)}<\infty$. On the other hand, $u_{N}\in\sobo^{2,p}(\Omega)$ for each $N\in\mathbb{N}$ and $\|u_{N}\|_{\sobo^{2,p}(\Omega)}\to\infty$ as $N\to\infty$. In view of \eqref{eq:contra}, this is the desired contradiction. Here, the key point is that $\mathbb{X}_{1}$ does not compactly embed into $\mathbb{X}_{0}$, which can also be checked by elementary means. 
\end{example}
\begin{figure}
\begin{tikzpicture}
\draw[-] (-0.1,0.25) -- (4,0.25); 
\draw[-] (-0.1,0.25) -- (-0.1,4);
\draw[-] (-0,1.5) -- (-0,0.5) -- (3.5,0.5) -- (3.5,1.5) -- (3,1.5) -- (3,2) -- (3.25,2) -- (3.25,3.5) -- (2.25,3.5) -- (2.25,2) -- (2.5,2) -- (2.5,1.5) -- (1.75,1.5) -- (1.75,1.75) -- (1.875,1.75) -- (1.875,2.5) -- (1.375,2.5)--(1.375,1.75) -- (1.5,1.75) -- (1.5,1.5) -- (1,1.5) -- (1,1.625) -- (1.065,1.625) -- (1.065,2) -- (0.825,2) -- (0.825,1.625) -- (0.89,1.625) -- (0.89,1.5) -- (0.75,1.5);
\draw[dotted] (0.75,1.5) -- (0,1.5);
\draw[<->] (3.25,3.75) -- (2.25,3.75);
\node[above] at (2.75,3.75){\tiny $2^{-j}$};
\draw [<->] (3.5,2) -- (3.5,3.5);
\node[right] at (3.5,2.75) {\tiny $2^{-j/2}$};
\draw[<->] (3.5,2) -- (3.5,1.5);
\node[right] at (3.5,1.75) {\tiny $2^{-j}$};
\node at (2.75,2.5) {$R_{j}$};
\node at (2.75,1.75) {\tiny $P_{j}$};
\node at (1.75,1) { $H$};
\end{tikzpicture}
\caption{Sketch of Maz'ya's shrinking rooms and passages domain from \cite[Chpt.~1.1.4, Fig.~3]{Mazya} (not to scale)}
\label{fig:mazya}
\end{figure}
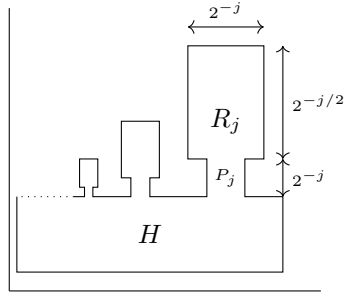

We now briefly discuss the impact of Dra\v{z}i\'{c}'s Lemma \ref{lem:drazic}. Starting with an inequality \eqref{eq:KornStart}, it improves the right-hand side of \eqref{eq:KornStart} to only contain the \emph{minimal} part $S_{1}(x)$ instead of $S_{0}(x)+S_{1}(x)$. In this sense, it reduces the amount of norms on the right-hand side needed to bound $\|x\|_{1}$. As shown by Example \ref{ex:MazyaDrazic}, if the embedding $\mathbb{X}_{1}\hookrightarrow\mathbb{X}_{0}$ is not compact, then Lemma \ref{lem:drazic} does not allow for this reduction. It is here where stronger additional statements need to be employed \emph{in conjunction with} an adaptation of the direct proof of Theorem \ref{thm:korn-gen} afterwards. However, we shall see below that a modification of the direct proof above \emph{directly} yields this conclusion. We illustrate this point in the following example. 

\begin{example}\label{ex:drazicredux}
Let $\Omega\subset\R^{n}$ be open and bounded with Lipschitz boundary, and let $\mathbb{A}$ be a first order $\mathbb{C}$-elliptic operator of the form \eqref{eq:form}. For $1\leq p<n$, we consider the Sobolev-type  inequality 
\begin{align}\label{eq:SobolevIntro}
\|\mathbf{u}\|_{\lebe^{\frac{np}{n-p}}(\Omega)}\leq c(\|\mathbf{u}\|_{\lebe^{p}(\Omega)}+\|\mathbb{A}\mathbf{u}\|_{\lebe^{p}(\Omega)})\qquad\text{for}\;\mathbf{u}\in\sobo^{1,p}(\Omega;V).
\end{align} 
We wish to replace $\|\mathbf{u}\|_{\lebe^{p}(\Omega)}$ by a norm $\vertiii{\cdot}$ on $\ker(\mathbb{A};\Omega)$. In the framework of Lemma \ref{lem:drazic}, we are bound to put $\mathbb{X}_{1}=\lebe^{\frac{np}{n-p}}(\Omega;V)$,  $\mathbb{X}_{0}=\lebe^{p}(\Omega;V)$ and $S_{0}(\mathbf{u})\coloneqq\|\mathbf{u}\|_{\lebe^{p}(\Omega)}$. However, neither is $S_{1}(\mathbf{u})\coloneqq\vertiii{\mathbf{u}}+\|\mathbb{A}\mathbf{u}\|_{\lebe^{p}(\Omega)}$ a sublinear functional on the entire $\mathbb{X}_{1}$ nor is the embedding $\mathbb{X}_{1}\hookrightarrow\mathbb{X}_{0}$ compact. In this case, {for $p>1$} the desired inequality can be obtained as follows: We first employ the stronger Korn-type inequality (see, e.g.,~\cite{DieningGmeineder,Kalamajska,Smith})
\begin{align}\label{eq:SobolevIntro1}
\|\mathbf{u}\|_{\sobo^{1,p}(\Omega)}\leq c(\|\mathbf{u}\|_{\lebe^{p}(\Omega)}+\|\mathbb{A}\mathbf{u}\|_{\lebe^{p}(\Omega)})\qquad\text{for}\;\mathbf{u}\in\sobo^{1,p}(\Omega;V), 
\end{align}
and subsequently we use Dra\v{z}i\'{c}'s Lemma \ref{lem:drazic} as above in conjunction with the Sobolev embedding $\sobo^{1,p}(\Omega;V)\hookrightarrow\lebe^{\frac{np}{n-p}}(\Omega;V)$. In conclusion, \eqref{eq:SobolevIntro} \emph{in itself} does not yield the desired reduction. 
\end{example}
In the preceding example, the key point is that we are interested in estimating lower order quantities. There are related scenarios where additional, stronger inequalities (such as \eqref{eq:SobolevIntro1} in the preceding example) are \emph{not} available; see Proposition \ref{prop:1} below for such an instance. 
\subsection{Variations of Theorem \ref{thm:korn-gen}}
We now consider stronger assumptions than in Theorem \ref{thm:korn-gen}. This particularly concerns the case $X=Y$ subject to the natural compact embedding hypothesis $\sobo^{k}X(\Omega)\hookrightarrow\hookrightarrow X(\Omega)$, which implies the following strengthening. 
\begin{theorem}\label{thm:variant}
In the situation of Theorem \ref{thm:korn-gen}, let $X=Y$ and suppose that $\sobo^{k}X(\Omega)\hookrightarrow\hookrightarrow X(\Omega)$. Then the following statements are \emph{equivalent}: 
\begin{enumerate}[label=\emph{(\roman*')}]
		\item\label{item:Korn1A1} We have 
		\begin{align}\label{eq:kornetto0A1}
			\begin{split}
				\norm{\bu}_{\sobo^{k}X(\Omega)} &\lesssim \norm{\bu}_{X(\Omega)} + \norm{\mathbb{A}\bu}_{X(\Omega)} \qquad\text{for all}\;\mathbf{u}\in\sobo^{k}X(\Omega;V). 
			\end{split}
		\end{align}
\item\label{item:Korn1B1} The following are equivalent for a seminorm $\vertiii{\cdot}$ on $\sobo^{k}X(\Omega;V)$: 
		\begin{enumerate}
			\item\label{item:Korn1aaa} $\vertiii{\cdot}$ is a norm on $\ker(\mathbb{A};\Omega)\cap\sobo^{k}X(\Omega;V)$. 
			\item\label{item:Korn2aaa} The following \emph{Korn-type inequality} holds 
			\begin{align}\label{eq:improvedKorn-v}
				\|\bu\|_{\sobo^{k}X(\Omega)} \lesssim \vertiii{\bu} + \|\mathbb{A}\bu\|_{X(\Omega)}\qquad\text{for all}\;\bu\in\sobo^{k}X(\Omega). 
			\end{align}
			\item\label{item:Korn3aaa} We have validity of the \emph{Korn-type inequality}
			\begin{align}\label{eq:improvedKorn1-v}
				\|\bu\|_{\sobo^{k}X(\Omega)} \lesssim  \|\mathbb{A}\bu\|_{X(\Omega)}\qquad\text{for all}\;\bu\in\sobo^{k}X(\Omega)\;\text{with}\;\vertiii{\bu}=0. 
			\end{align}
		\end{enumerate}
		\end{enumerate} 
Moreover, if \ref{item:Korn1A1} or \ref{item:Korn1B1} hold, then we automatically have 
\begin{align}\label{eq:nullspaceregularity}
\ker(\mathbb{A};\Omega)=\ker(\mathbb{A};\Omega)\cap \sobo^{k}X(\Omega;V) = \ker(\mathbb{A};\Omega)\cap\hold^{\infty}(\overline{\Omega};V)
\end{align}
and 
\begin{align}\label{eq:poincaredrazic}
    	\inf_{\bm{\rho}\in\ker(\mathbb{A};\Omega)}\norm{\bu-\bm{\rho}}_{X(\Omega)} & \lesssim \|\mathbb{A}\bu\|_{X(\Omega)}\qquad\text{for all}\;\mathbf{u}\in\sobo^{k}X(\Omega;V). 
\end{align} 
\end{theorem}
Comparing assertion \emph{\ref{item:Korn1A1}} and \emph{\ref{item:Korn1A}} from Theorem \ref{thm:korn-gen} for $X=Y$, the sole difference lies between admitting elements $\bm{\rho}\in\ker(\mathbb{A};\R^{n})$ and $\bm{\rho}\in\ker(\Omega;\R^{n})$. In particular, subject to \eqref{eq:kornetto0A1}, $\Omega$ need not be connected. The proof below uses Dra\v{z}i\'{c}'s Lemma \ref{lem:drazic}, but can alternatively be obtained by means of the direct approach discussed in Section \ref{sec:directproof} if one aims to track constants. 

Additionally, for the proof of Theorem \ref{thm:variant} we require the classical Peetre-Tartar lemma. 
\begin{lemma}[Peetre-Tartar lemma, {\cite[Lem.~11.1]{Tartar}}]\label{lem:PT} Let $(\mathbb{Y}_{i},\|\cdot\|_{i})$, $i\in\{0,1,2\}$, be three Banach spaces. Moreover, let $S\colon \mathbb{Y}_{0}\to\mathbb{Y}_{1}$ be linear and compact, and let $T\colon\mathbb{Y}_{0}\to\mathbb{Y}_{2}$ be linear and bounded. If $x\mapsto \|Sx\|_{1}+\|Tx\|_{2}$ is a norm on $\mathbb{Y}_{0}$ which is equivalent to $\|\cdot\|_{0}$, then one as that $\dim(\ker(T))<\infty$.
\end{lemma}
\begin{proof}[Proof of Theorem \ref{thm:variant}]
On '\emph{\ref{item:Korn1A1}$\Rightarrow$\ref{item:Korn1B1}}'. By assumption, we have the compact embedding 
\begin{align*}
\mathbb{Y}_{0}\coloneqq \sobo^{k}X(\Omega;V)\hookrightarrow\hookrightarrow \mathbb{Y}_{1}\coloneqq X(\Omega;V).
\end{align*}
In the framework of Lemma~\ref{lem:PT}, we moreover put $\mathbb{Y}_{2}\coloneqq X(\Omega;W)$, $S$ the identity map, and  $T\coloneqq\mathbb{A}$ is bounded. Then  $\eqref{eq:kornetto0A1}$ and Lemma~\ref{lem:PT} imply that $\dim(\ker(\mathbb{A};\Omega)\cap\sobo^{k}X(\Omega;V))<\infty$. 

On '\ref{item:Korn1aaa}$\Rightarrow$\ref{item:Korn2aaa}'. In the present situation, we set $\mathbb{X}_{0}\coloneqq X(\Omega;V)$, $\mathbb{X}_{1}\coloneqq \sobo^{k}X(\Omega;V)$, and 
\begin{align*}
S_{0}(\mathbf{u})\coloneqq\|\mathbf{u}\|_{X(\Omega)}\;\;\;\text{and}\;\;\;S_{1}(\mathbf{u})\coloneqq\vertiii{\mathbf{u}}+\|\mathbb{A}\mathbf{u}\|_{X(\Omega)},\qquad\mathbf{u}\in\sobo^{k}X(\Omega;V). 
\end{align*}
Then the inequality $\|\mathbf{u}\|_{\sobo^{k}X(\Omega)}\leq S_{0}(\mathbf{u})+S_{1}(\mathbf{u})$ for $\mathbf{u}\in\sobo^{k}X(\Omega;V)$ follows from \emph{\ref{item:Korn1A1}}. Now, \ref{item:Korn1aaa} implies that $\ker(S_{1})=\{0\}$. Because of $\mathbb{X}_{1}\hookrightarrow\hookrightarrow\mathbb{X}_{0}$, Dra\v{z}i\'{c}'s Lemma~\ref{lem:drazic} yields~\ref{item:Korn2aaa}. The remaining directions '\ref{item:Korn2aaa}$\Rightarrow$\ref{item:Korn3aaa}$\Rightarrow$\ref{item:Korn1aaa}' are trivial. 

For '\emph{\ref{item:Korn1B1}$\Rightarrow$\ref{item:Korn1A1}}', it suffices to note that $\|\cdot\|_{X(\Omega)}$ is a norm on $\ker(\mathbb{A};\Omega)\cap\sobo^{k}X(\Omega;V)$. In particular,  \ref{item:Korn1aaa} is fulfilled with $\vertiii{\cdot}=\|\cdot\|_{X(\Omega)}$. In consequence, \emph{\ref{item:Korn1A1}} follows from~\ref{item:Korn2aaa}. Summarising, \emph{\ref{item:Korn1A1}} and \emph{\ref{item:Korn1B1}} are equivalent. 

To prove \eqref{eq:nullspaceregularity}, we claim that \emph{\ref{item:Korn1A1}} or equivalently \emph{\ref{item:Korn1B1}} imply that $\Omega$ has only finitely many connected components. To see this, recall that if~\emph{\ref{item:Korn1A1}} holds, then by the Peetre--Tartar Lemma~\ref{lem:PT} one has that $m \coloneqq \dim(\ker(\mathbb{A};\Omega)\cap\sobo^{k}X(\Omega;V))<\infty$. Now suppose that $\Omega$ has $M>m$ connected components, so that we may write $\Omega=(\Omega_{1}\cup...\cup\Omega_{M})\cup U_{M+1}$, where $\Omega_{1},...,\Omega_{M},U_{M+1}$ are open and pairwise disjoint. As a consequence of Lemma \ref{lem:Celliptic} and because $M$ is finite, we have 
\begin{align}\label{eq:isomorphy}
\mathscr{X}_{M}\coloneqq\ker\Big(\mathbb{A};\bigcup_{j=1}^{M}\Omega_{j}\Big)\cap\sobo^{k}X\Big(\bigcup_{j=1}^{M}\Omega_{j};V\Big) = \ker\Big(\mathbb{A};\bigcup_{j=1}^{M}\Omega_{j}\Big)\cong \bigoplus_{j=1}^{M} \ker(\mathbb{A};\Omega_{j}).
\end{align}
Every $\ker(\mathbb{A};\Omega_{j})$ is at least one-dimensional, since it contains the constants. Therefore, 
\begin{align*}
m < M \leq \dim\Big(\bigoplus_{j=1}^{M}\ker(\mathbb{A};\Omega_{j})\Big) = \dim(\mathscr{X}_{M}) \leq \dim(\ker(\mathbb{A};\Omega)\cap\sobo^{k}X(\Omega;V))=m. 
\end{align*}
This is a contradiction, and so $\Omega$ can have at most $M\leq m$ connected components. This directly proves~\eqref{eq:nullspaceregularity}, and then Dra\v{z}i\'{c}'s Lemma \ref{lem:drazic} applied to $S_1(\bu) \coloneqq \|\mathbb{A} \bu \|_{X(\Omega)}$ and $S_0(\bu) \coloneqq \norm{\bu}_{X(\Omega)}$, see \eqref{eq:Drazic1}, implies \eqref{eq:poincaredrazic}. The proof is complete.
\end{proof}
As a by-product of the preceding proof, we record: 
\begin{corollary}\label{cor:connectedness}
In the situation of Theorem \ref{thm:korn-gen}, let $X=Y$ and suppose that $\sobo^{k}X(\Omega)\hookrightarrow\hookrightarrow X(\Omega)$. Then the validity of the Korn-type inequality 
\begin{align*}
\|\mathbf{u}\|_{\sobo^{k}X(\Omega)}\lesssim \|\mathbf{u}\|_{X(\Omega)} + \|\mathbb{A}\mathbf{u}\|_{X(\Omega)}\qquad\text{for all}\;\mathbf{u}\in\sobo^{k}X(\Omega)
\end{align*}
implies that $\Omega$ has at most \emph{finitely many connected components}. 
\end{corollary}
\begin{remark}[Explicit constants]\label{rem:explicitconstants}
  {In certain situations explicit constants in the Korn inequalities or upper bounds are available, see \cite{BauerPauly,PayneWeinberger,DesvilletesVillani} for an incomplete list. In this regard, the alternative direct can be employed to yield specific dependencies on the underlying constants.
  }
  \end{remark}
We conclude this section with an instance where Dra\v{z}i\'{c}'s Lemma \ref{lem:drazic} does not apply, whereas the direct proof from Section~\ref{sec:directproof} yields the desired inequality. This concerns the case $p=1$, where \eqref{eq:SobolevIntro} holds \cite{GmeinederRaita}, but \eqref{eq:SobolevIntro1} does not; the latter is a well-known principle usually referred to as \emph{Ornstein's Non-Inequality} \cite{Ornstein}; see also \cite{KirchheimKristensen} for the general context considered here. For instance, we have the following 
\begin{proposition}\label{prop:1}
Let $\Omega\subset\R^{n}$ be open and bounded, and let $\mathbb{A}$ be a first order $\mathbb{C}$-elliptic operator of the form \eqref{eq:form}. We define 
\begin{align*}
\mathrm{BV}^{\mathbb{A}}(\Omega)\coloneqq \{\mathbf{u}\in\lebe^{1}(\Omega;V)\colon\;\mathbb{A}\mathbf{u}\;\text{is a finite $W$-valued Radon measure}\}
\end{align*}
and denote by $|\mathbb{A}\mathbf{u}|(\Omega)$ the total variation of $\mathbb{A}\mathbf{u}$. Moreover, suppose that $\Omega$ is such that 
\begin{align}\label{eq:Sobolevp=1}
\begin{split}
\|\mathbf{u}\|_{\lebe^{\frac{n}{n-1}}(\Omega)} & \lesssim \|\mathbf{u}\|_{\lebe^{1}(\Omega)} + |\mathbb{A}\mathbf{u}|(\Omega)\qquad \text{and} \\ 
\inf_{\bm{\rho}\in\ker(\mathbb{A};\R^{n})}\|\mathbf{u}-\bm{\rho}\|_{\lebe^{1}(\Omega)}& \lesssim |\mathbb{A}\mathbf{u}|(\Omega)
\end{split}
\end{align}
hold for all $\mathbf{u}\in\mathrm{BV}^{\mathbb{A}}(\Omega)$ and let $\vertiii{\cdot}$ be a seminorm on $\mathrm{BV}^{\mathbb{A}}(\Omega)$. Then the inequality 
\begin{align}\label{eq:conclude1}
\|\mathbf{u}\|_{\lebe^{\frac{n}{n-1}}(\Omega)}\lesssim \vertiii{\mathbf{u}} + |\mathbb{A}\mathbf{u}|(\Omega)\qquad\text{for all}\;\mathbf{u}\in\mathrm{BV}^{\mathbb{A}}(\Omega)
\end{align}
holds if and only if $\vertiii{\cdot}$ is a norm on $\ker(\mathbb{A};\R^{n})|_{\Omega}$. 
\end{proposition}
\begin{proof}
This follows analogously as in the direct proof of Theorem \ref{thm:korn-gen} from Section \ref{sec:directproof}. Retaining the conventions of this proof, we have as a substitute of \eqref{eq:intermediate1}: 
\begin{align*}
\|\mathbf{u}\|_{\lebe^{\frac{n}{n-1}}(\Omega)} & \leq \|\mathbf{u}-\bm{\rho}_{0}\|_{\lebe^{\frac{n}{n-1}}(\Omega)}+\|\bm{\rho}_{0}\|_{\lebe^{\frac{n}{n-1}}(\Omega)} \\ 
& \lesssim \|\mathbf{u}-\bm{\rho}_{0}\|_{\lebe^{1}(\Omega)}+|\mathbb{A}\mathbf{u}|(\Omega) + \|\bm{\rho}_{0}\|_{\lebe^{\frac{n}{n-1}}(\Omega)} \\ 
& \lesssim |\mathbb{A}\mathbf{u}|(\Omega) + \|\bm{\rho}_{0}\|_{\lebe^{\frac{n}{n-1}}(\Omega)} \eqqcolon A + B. 
\end{align*}
Now, subject to $\vertiii{\cdot}$ being a norm on $\ker(\mathbb{A};\R^{n})|_{\Omega}$, we have that 
\begin{align*}
B & = \|\bm{\rho}_{0}\|_{\lebe^{\frac{n}{n-1}}(\Omega)} = \|\Pi(\bm{\rho}_{0})\|_{\lebe^{\frac{n}{n-1}}(\Omega)}\\ 
& \lesssim \|\Pi(\bm{\rho}_{0})\|_{\lebe^{1}(\Omega)} = \|\Pi(\mathbf{u}-\bm{\rho}_{0})\|_{\lebe^{1}(\Omega)} + \|\Pi(\mathbf{u})\|_{\lebe^{1}(\Omega)} \\ 
& \lesssim \|\mathbf{u}-\bm{\rho}_{0}\|_{\lebe^{1}(\Omega)} + \vertiii{\mathbf{u}} \lesssim |\mathbb{A}\mathbf{u}|(\Omega)+\vertiii{\mathbf{u}}. 
\end{align*}
Combining the previous estimates, \eqref{eq:conclude1} follows. Conversely, if $\vertiii{\cdot}$ is a seminorm which fails to be a norm, there exists an element $\mathbf{u}\in\mathrm{BV}^{\mathbb{A}}(\Omega)$ for which the right-hand side of \eqref{eq:conclude1} vanishes, whereas the left-hand side does not. This completes the proof. 
\end{proof}
Note that Dra\v{z}i\'{c}'s Lemma \ref{lem:drazic} cannot be applied here for several reasons; besides the mismatching Banach space set-up, the embedding $\lebe^{\frac{n}{n-1}}(\Omega)\hookrightarrow\lebe^{1}(\Omega)$ is never compact. 
\begin{remark}
The spaces $\mathrm{BV}^{\mathbb{A}}(\Omega)$ from~\cite{BreitDieningGmeineder} yield the classical space $\mathrm{BV}(\Omega)$ for $\mathbb{A}=\nabla$ and the space $\mathrm{BD}(\Omega)$ of functions of bounded formation provided that $\mathbb{A}=\varepsilon$. In particular, Proposition \ref{prop:1} applies to these special cases. 
\end{remark}
\begin{remark}
The validity of~\eqref{eq:Sobolevp=1} is always ensured provided that $\Omega$ is an open and bounded Lipschitz domain or, more generally, an $(\varepsilon,\delta)$-domain in the sense of Jones~\cite{Jones}; see Section \ref{sec:validityfirst} below for this terminology. 
More precisely, $\eqref{eq:Sobolevp=1}_{1}$ follows from~\cite{GmeinederRaita}, in turn localising the ellipticity and cancellation condition from \cite{VanSchaftingen} to domains. Inequality $\eqref{eq:Sobolevp=1}_{2}$ follows, e.g., from~\cite{BreitDieningGmeineder,DieningGmeineder,Kalamajska}.
\end{remark}

\section{Bulk and boundary trace conditions}\label{sec:choices0}
In this section, we collect specific instances of Korn-type inequalities where the seminorms $\vertiii{\cdot}$ are given by integrals over sets of positive Lebesgue measure or boundary trace integrals. Table \ref{table:InequalitiesSummary} provides a summary of the underlying inequalities to be established in this and the following section.

\subsection{Validity of Korn's first inequality}\label{sec:validityfirst}
Both Theorems~\ref{thm:korn-gen} and~\ref{thm:variant} establish variants of Korn's inequality \emph{subject} to the validity of Korn's first inequality, see   $\eqref{eq:kornetto0}_{1}$ and \eqref{eq:kornetto0A1}, respectively. The latter requires conditions on both the underlying function spaces, the domains and the differential operators. In view of the examples below in Sections~\ref{sec:bulkcond}--\ref{sec:choices}, we now collect various conditions on the domains which lead to the validity of Korn's first inequality.

To this end, it is customary to define for a $\mathbb{C}$-elliptic operator of the form \eqref{eq:form} and an open  set $\Omega\subset\R^{n}$
\begin{align}
\sobo^{\mathbb{A}}X(\Omega)\coloneqq\{u\in \sobo^{k-1}X(\Omega;V)\colon\;\mathbb{A}u\in X(\Omega;W)\},
\end{align}
and to endow it with the natural norm $\|\mathbf{u}\|_{\sobo^{\mathbb{A}}X(\Omega)}\coloneqq\|\mathbf{u}\|_{\sobo^{k-1}X(\Omega)}+\|\mathbb{A}\mathbf{u}\|_{X(\Omega)}$. We shall consider open and bounded domains of the following types:

\begin{table}
\begin{tabular}{c||c|c|c|c}
 & Bulk & Full boundary  & Partial boundary & Lower dimensional \\
{{\large $\mathbb{A}$}} & \hspace{0.3cm} conditions\hspace{0.3cm} & traces & traces & traces\\
\hline 
\hline
 $\nabla^{D}$ &
  \multirow{4}{5em}{
\(\left\{\rule{0pt}{2.9em}\right.\)  
 }\hspace{-36pt}
 \multirow{4}{5em}{Prop.~\ref{prop:bulk} Prop.~\ref{prop:Orlicz}} &  
 \multirow{4}{5em}{
\(\left\{\rule{0pt}{2.9em}\right.\)  
 }\hspace{-36pt}
 \multirow{4}{5em}{Prop.~\ref{prop:fullbdrytrace}
 Prop. \ref{prop:fullp=1} Prop.~\ref{prop:Orlicz}} & Prop.~\ref{prop:devgradpartial} & Ex.~\ref{ex:devgradmu} \\
$\sg$ &  & & Ex.~\ref{prop:partialtracesymgrad} & Ex.~\ref{ex:mutracesymgrad1}, \ref{ex:symgradn=3}\\
$\sg^{D}$ &  &  & Ex.~\ref{ex:devsymgradcomplications} & Ex.~\ref{ex:devgrad}\\
general &  &  & Ex.~\ref{ex:devsymgradcomplications} & Ex.~\ref{ex:devgrad}
\end{tabular}
\vspace{0.5cm}
\caption{Summary of the inequalities established in Sections \ref{sec:choices0} and \ref{sec:choices}.}\label{table:InequalitiesSummary}
\end{table}

\begin{itemize}
 \item \emph{Domains with Lipschitz or $\hold^{k}$-boundaries.} As usual, we say that $\Omega\subset\R^{n}$ has Lipschitz or $\hold^{k}$-boundary if the following holds for every point $x_{0}\in\partial\Omega$: After a rotation and translation, there exists open and bounded sets $U\subset\R^{n-1}$, $V\subset\R^{n}$ and a Lipschitz or $\hold^{k}$-function $f\colon U\to \R$ such that $x_{0}\in\mathrm{graph}(f)$,  $\partial\Omega\cap V=\mathrm{graph}(f)$, $\Omega\cap V$ lies strictly below $\mathrm{graph}(f)$ and $(\R^{n}\setminus\overline{\Omega})\cap V$ lies strictly above $\mathrm{graph}(f)$. Moreover, we suppose that  the Lipschitz or $\hold^{k}$-norms of such functions $f$ to be uniformly bounded.
	\item \emph{Extension domains for $\sobo^{\mathbb{A}}X$.} If there exists a bounded linear extension operator $E\colon \sobo^{\mathbb{A}}X(\Omega)\to\sobo^{\mathbb{A}}X(\R^{n})$ (so, in particular, $Eu|_{\Omega}=u$ for all $u\in\sobo^{\mathbb{A}}X(\Omega)$), $\Omega$ is called an extension domain for $\sobo^{\mathbb{A}}X$.
	\item \emph{Jones $(\varepsilon,\delta)$-domains.} A bounded domain $\Omega\subset\R^{n}$ is an $(\varepsilon,\delta)$-domain with $0<\varepsilon,\delta<\infty$ if it satisfies the following property: If $x,y\in\Omega$ are such that $|x-y|<\delta$, then there exists a curve $\gamma$ connecting $x$ and $y$ such that 
	\begin{align}\label{eq:Jones}
	\ell(\gamma)\leq \frac{|x-y|}{\varepsilon}\quad\text{and}\quad\mathrm{dist}(z,\partial\Omega)\geq \varepsilon\min\{|x-z|,|y-z|\},
	\end{align}
	for all elements $z$ on the curve $\gamma$. This means that each two sufficiently close points can be joined by a curve with length bounded by their distance, and every point of this curve is at a sufficient distance from $\partial\Omega$. In particular, the curve can be surrounded by a banana-like domain with uniform parameters which still lies inside $\Omega$, see Figure \ref{fig:domains}(i). This condition is violated, e.g., in the cusp example in Figure \ref{fig:domains}(ii): Fixing $x_{0}$ and considering points $y_{0}$ very close to the cusp, it is not possible to find such uniform banana-like domains. Conversely, by considering points $x_{0},y_{0}$ close to the slit   in Figure \ref{fig:domains}(iii), the first condition in $\eqref{eq:Jones}$  is violated. 
	\item \emph{John domains.} A bounded domain $\Omega\subset\R^{n}$ is \emph{John} if there exists a distinguished point $z_{0}\in\Omega$ and a constant $c_{J}>0$ with the following property: For any $x\in\Omega$, there exists a curve $\gamma\colon[0,\ell(\gamma)]\to\Omega$ parametrised by arc-length such that $\gamma(0)=x$, $\gamma(\ell(\gamma))=z_{0}$ and $\mathrm{dist}(\gamma(t),\partial\Omega)\geq c_{J}t$ for all $t\in[0,\ell(\gamma)]$. For instance, the slit domain in Figure~\ref{fig:domains}(iii) or the twisted brick example in Figure~\ref{fig:domains}(iv) are John.
\end{itemize}
Since polyhedral domains do not necessarily have Lipschitz boundary (see Figure \ref{fig:domains}(iv)), John domains are particularly important in applications in numerical analysis; 
there, polyhedral domains appear as domains, as approximations of domains with curved boundaries, and as neighbourhoods of polytopal partitions.   
\begin{remark}[Relations and properties]\label{rem:relations} The aforementioned conditions are related as follows: 
	\begin{itemize}
        \item[(i)] Every domain with Lipschitz or $\hold^{k}$-boundary is an $(\varepsilon,\delta)$-domain, and every $(\varepsilon,\delta)$-domain is an extension domain for $\sobo^{\mathbb{A}}\lebe^{p}$. The latter follows, e.g., from \cite{GmeinederRaita} based on  an adaptation of Jones' original method for Sobolev spaces \cite{Jones}. Moreover, in all these cases, we have the compact embedding $\sobo^{k,p}(\Omega)\hookrightarrow\hookrightarrow\lebe^{p}(\Omega)$. 
		\item[(ii)] Open and bounded domains with Lipschitz or $\hold^{k}$-boundary are John. John domains may have slits and thus are, in general, not extension domains for $\sobo^{1,p}$; in particular, since every $(\varepsilon,\delta)$-domain is a $\sobo^{1,p}$-extension domain, John domains are not necessarily $(\varepsilon,\delta)$-domains.  Moreover, by \cite[Thm.~2.2]{Chuaetal}, we have the compact embedding $\sobo^{1,p}(\Omega)\hookrightarrow\hookrightarrow\lebe^{p}(\Omega)$ for John domains. 
	\end{itemize}

\end{remark}
\begin{figure}[t]
\begin{tikzpicture}
\clip (-1.6,0.86) rectangle (7.9,2.6);
\begin{scope}{xshift=4cm}
\clip(-1.5,0.8) rectangle (7.6,2.5);
\draw[-, rounded corners =5pt] (-2,2) -- (-1,1) -- (0,1.5) -- (1,1) [out= 0, in = 240] to (2,2.5);
\draw[-,fill=black!20!white, rounded corners =5pt] (-2.5,2.5) -- (-2,2) -- (-1,1) -- (0,1.5) -- (1,1) [out= 0, in = 240] to (2,2.5) ;
\draw[-,dotted,white,thick,fill=white!90!black] (-1,1.5) [out =20, in = 150] to (0.85,1.65) [out=80, in =60] to (-1,1.5);
\node at (-1,1.5) {{\tiny\textbullet}};
\node[below] at (-1,1.5) {{\tiny $x_{0}$}};
\node at (0.85,1.65) {{\tiny\textbullet}};
\node[below] at (0.85,1.65) {{\tiny $y_{0}$}};
\draw[-,dotted] (-1,1.5) [out =35, in = 120] to (0.85,1.65);
\draw[-,fill=black!20!white] (2.5,1)  [out=20,in=270] to (3.5,2.5) [out=-90, in =160] to (4.5,1);
\node at (3.2,1.25) {{\tiny\textbullet}};
\draw[dotted,-] (3.2,1.25) [out=40, in =270] to (3.5,2.1);
\node[right] at (3.2,1.25) {{\tiny $x_{0}$}};
\draw[-,fill=black!20!white] (5.5,2.5) --(5.5,1) -- (7.5,1) -- (7.5,2.5);
\draw[-] (6.5,1) -- (6.5,1.85);
\end{scope}
\end{tikzpicture}
\begin{tikzpicture}[scale=0.45, x={(1cm,0cm)}, y={(0.5cm,0.3cm)}, z={(0cm,1cm)}]

\node[ fill=white, circle ] at (-0.7,0.2) {};

\def\length{4}    
\def\width{1}     
\def\height{1.2}  

\coordinate (A) at (0,0,0);
\coordinate (B) at (\length,0,0);
\coordinate (C) at (\length,\width,0);
\coordinate (D) at (0,\width,0);

\coordinate (E) at (0,0,\height);
\coordinate (F) at (\length,0,\height);
\coordinate (G) at (\length,\width,\height);
\coordinate (H) at (0,\width,\height);

\fill[gray!20] (E) -- (F) -- (G) -- (H) -- cycle; 
\fill[gray!40] (A) -- (B) -- (F) -- (E) -- cycle; 
\fill[gray!30] (B) -- (C) -- (G) -- (F) -- cycle; 

\draw[black] (A) -- (B) -- (F) -- (E) -- cycle;
\draw[black] (B) -- (C) -- (G) -- (F);
\draw[black] (E) -- (H) -- (G);  

\def\topx{\width}     
\def\topy{\length}    
\def\topz{\height}

\def\zshift{\height}
\pgfmathsetmacro{\xoffset}{(\length - \topx)/2}
\pgfmathsetmacro{\yoffset}{(\width - \topy)/2}

\coordinate (A2) at (\xoffset,\yoffset,\zshift);
\coordinate (B2) at (\xoffset,\yoffset+\topy,\zshift);
\coordinate (C2) at (\xoffset+\topx,\yoffset+\topy,\zshift);
\coordinate (D2) at (\xoffset+\topx,\yoffset,\zshift);

\coordinate (E2) at (\xoffset,\yoffset,\zshift+\topz);
\coordinate (F2) at (\xoffset,\yoffset+\topy,\zshift+\topz);
\coordinate (G2) at (\xoffset+\topx,\yoffset+\topy,\zshift+\topz);
\coordinate (H2) at (\xoffset+\topx,\yoffset,\zshift+\topz);

\fill[gray!30] (D2) -- (C2) -- (G2) -- (H2) -- cycle; 
\fill[gray!20] (G2) -- (H2) -- (E2) -- (F2) -- cycle; 
\fill[gray!40] (A2) -- (D2) -- (H2) -- (E2) -- cycle;  

\draw[black] (A2) -- (D2);       
\draw[black] (E2) -- (A2);
\draw[black] (D2) -- (H2);      
\draw[black] (E2) -- (H2);     
\draw[black] (F2) -- (G2);      
\draw[black] (E2) -- (F2);     
\draw[black] (C2) -- (D2);     
\draw[black] (C2) -- (G2);      
\draw[black] (H2) -- (G2);

\end{tikzpicture}
\begin{tikzpicture}[xshift=0.5cm]
\node at (1.75,-1) {(i)};
\node at (5.25,-1) {(ii)};
\node at (8.5,-1) {(iii)};
\node at (11.5,-1) {(iv)};
\end{tikzpicture}
\caption{Conditions on domains. (i) reflects Jones' $(\varepsilon,\delta)$-condition \eqref{eq:Jones}. It is not satisfied in (ii) and (iii); still,  the slit domain from (iii) is John. Similarly, (iv) is a typical instance of a polyhedral non-Lipschitz, yet John domain.}\label{fig:domains}
\end{figure}
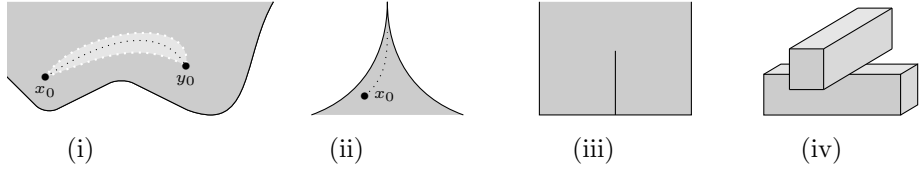
For (bounded) John domains, Korn inequalities have been established in \cite{AcostaDuranMuschietti,DieningRuzickaSchumacher,Kauranen} for the symmetric gradient, and we refer the reader to \cite{DieningGmeineder} for general $\mathbb{C}$-elliptic operators. The general outcome is that, for an open and bounded John domain $\Omega\subset\R^{n}$ and a $k$-th order $\mathbb{C}$-elliptic operator of the form \eqref{eq:form}, we have the estimate 
\begin{align}\label{eq:KornJohn}
\|\mathbf{u}\|_{\sobo^{k,p}(\Omega)}\lesssim \|\mathbf{u}\|_{\lebe^{p}(\Omega)}+\|\mathbb{A}\mathbf{u}\|_{\lebe^{p}(\Omega)}\qquad\text{for all}\;\mathbf{u}\in\sobo^{k,p}(\Omega;V). 
\end{align}
This result persists for other spaces such as certain Orlicz or Muckenhoupt weighted Lebesgue spaces; see \cite[Thm.~6.9, Cor.~6.14]{DieningGmeineder} and the discussion in Section \ref{sec:otherscales} below. 

In the framework of extension domains for $\sobo^{\mathbb{A}}X$, Korn's first inequality can be obtained by easy means provided that the Banach function space $X$ supports Mihlin's theorem:
\begin{assumption}\label{ass:Mihlin}
	Let $X$ be a Banach function space which is definable on all open subsets of $\R^{n}$. We say that $X$ \emph{supports Mihlin's multiplier theorem} if, for any $\mathfrak{m}\in\hold^{\infty}(\R^{n}\setminus\{0\};\mathbb{C})$ which is homogeneous of degree zero, there exists $c>0$ such that the associated Fourier multiplier operator $Tu\coloneqq\mathscr{F}^{-1}(\mathfrak{m}\widehat{u})$ satisfies 
	\begin{align*}
		\|T(u)\|_{X(\R^{n})}\leq c\|u\|_{X(\R^{n})}\qquad\text{for all}\;u\in\hold_{c}^{\infty}(\R^{n};\mathbb{C}). 
	\end{align*}
\end{assumption}
Subject to Assumption \ref{ass:Mihlin}, we then have: 
\begin{corollary}
Let $\mathbb{A}$ be a $k$-th order $\mathbb{C}$-elliptic operator of the form \eqref{eq:form}. Moreover, suppose that $X$ satisfies Assumption \ref{ass:Mihlin} and is such that $\Omega$ is an extension domain for $\sobo^{\mathbb{A}}X$. If $\hold_{c}^{\infty}(\Omega;V)$ is dense in $(\sobo^{\mathbb{A}}X(\R^{n}),\|\cdot\|_{\sobo^{\mathbb{A}}X(\R^{n})})$ and $\|\mathbf{u}\|\coloneqq \|\mathbf{u}\|_{X(\R^{n})}+\|D^{k}\mathbf{u}\|_{X(\R^{n})}$ is an equivalent norm on $\sobo^{k}X(\R^{n})$, then there holds 
\begin{align}\label{eq:Kornonextdomains}
\|\mathbf{u}\|_{\sobo^{k}X(\Omega)}\lesssim \|\mathbf{u}\|_{\sobo^{k-1}X(\Omega)}+\|\mathbb{A}\mathbf{u}\|_{X(\Omega)}\qquad\text{for all}\;\mathbf{u}\in\sobo^{k}X(\Omega;V). 
\end{align}
\end{corollary}
\begin{proof}
Denote by $E\colon\sobo^{\mathbb{A}}X(\Omega)\to\sobo^{\mathbb{A}}X(\R^{n})$ a bounded linear extension operator. Since $\mathbb{A}$ is $\mathbb{C}$-elliptic, it is in particular elliptic in the sense that $\mathbb{A}[\xi]\colon V\to W$ is injective for all $\xi\in\R^{n}\setminus\{0\}$. Denoting by $\mathbb{A}^{*}[\xi]\colon W'\simeq W\to V\simeq V'$ its adjoint, we have that $\mathbb{A}^{*}[\xi]\mathbb{A}[\xi]\colon V\to V$ is bijective for any $\xi\in\R^{n}\setminus\{0\}$. Hence, we may write for any $\mathbf{v}\in\hold_{c}^{\infty}(\R^{n};V)$ and all $|\alpha|=k$: 
\begin{align}\label{eq:rep}
	\partial^{\alpha}\mathbf{v} = \mathscr{F}^{-1}(\mathfrak{m}_{\alpha}(\xi)\widehat{\mathbb{A}\mathbf{v}}) \coloneqq  \mathscr{F}^{-1}(\xi^{\alpha}(\mathbb{A}[\xi]^{*}\mathbb{A}[\xi])^{-1}\mathbb{A}[\xi]\widehat{\mathbb{A}\mathbf{v}}).
\end{align} 
We note that $\mathfrak{m}_{\alpha}\in\hold^{\infty}(\R^{n}\setminus\{0\};\mathscr{L}(W;V))$ is homogeneous of degree zero, where $\mathscr{L}(W;V)$ denotes the linear operators from $W$ to $V$. The latter is a finite-dimensional space, whereby Assumption \ref{ass:Mihlin} and \eqref{eq:rep}   imply that $\|\partial^{\alpha}\mathbf{v}\|_{X(\R^{n})}\lesssim \|\mathbb{A}\mathbf{v}\|_{X(\R^{n})}$ for all $\mathbf{v}\in\hold_{c}^{\infty}(\R^{n};V)$. 
We use the smooth approximation property to find that this inequality also holds true for all $\mathbf{v}\in\sobo^{\mathbb{A}}X(\R^{n})$. For $\mathbf{u}\in\sobo^{\mathbb{A}}X(\Omega)$, we conclude that 
\begin{align*}
\|\mathbf{u}\|_{\sobo^{k}X(\Omega)} & \leq \|E\mathbf{u}\|_{\sobo^{k}X(\R^{n})} \lesssim \|E\mathbf{u}\|_{X(\R^{n})}+\|D^{k}E\mathbf{u}\|_{X(\R^{n})} \\ 
& \lesssim \|E\mathbf{u}\|_{X(\R^{n})} + \|\mathbb{A}E\mathbf{u}\|_{X(\R^{n})} 
\leq \|E\mathbf{u}\|_{\sobo^{\mathbb{A}}X(\R^{n})} \lesssim \|\mathbf{u}\|_{\sobo^{\mathbb{A}}X(\Omega)}. 
\end{align*}
This is \eqref{eq:Kornonextdomains}, and the proof is complete. 
\end{proof}
In the following, we shall consider specific regularities of domains. However, based on the above summary, the reader might verify whether the underlying approaches also extend to the setting required in the respective application. 
\subsection{Bulk conditions}\label{sec:bulkcond} As a first example, we begin with bulk conditions on a subset $\omega \subset \Omega$. Here, the key point is that $\omega$ might have very small Lebesgue measure compared to $\Omega$. 
\begin{proposition}[Bulk conditions]\label{prop:bulk}
Let $\mathbb{A}$ be a $k$-th order $\mathbb{C}$-elliptic differential operator of the form \eqref{eq:form}. Moreover, given $1<p<\infty$, suppose that the open, bounded and connected set $\Omega\subset\R^{n}$ supports the first Korn inequality \eqref{eq:kornetto0A1} and is such that $\sobo^{k,p}(\Omega)\hookrightarrow\hookrightarrow\lebe^{p}(\Omega)$. Then for any $\mathscr{L}^{n}$-measurable subset $\omega\subset\Omega$ with $\mathscr{L}^{n}(\omega)>0$, there exists a constant $c>0$ such that 
\begin{align}\label{eq:bulk}
\|\mathbf{u}\|_{\sobo^{k,p}(\Omega)}\leq c\Big(\|\mathbf{u}\|_{\lebe^{1}(\omega)} + \|\mathbb{A}\mathbf{u}\|_{\lebe^{p}(\Omega)} \Big)\qquad\text{for all}\;\mathbf{u}\in\sobo^{k,p}(\Omega;V). 
\end{align}
Moreover, writing $c=c(\omega)$, we have $c(\omega)\to\infty$ as $\mathscr{L}^{n}(\omega)\searrow 0$. In particular, if $k=1$, \eqref{eq:bulk} holds for $\mathbb{A}\in\{\nabla^{D},\varepsilon\}$ if $n\geq 2$ and $\mathbb{A}=\sg^{D}$ if $n\geq 3$. 
\end{proposition}
By our discussion from Section \ref{sec:validityfirst}, the hypotheses of the preceding proposition are satisfied for all Lipschitz, $\hold^{k}$-, Jones $(\varepsilon,\delta)$- and John domains. 
\begin{proof}
By Theorem \ref{thm:variant}, we only have to verify that $\vertiii{\mathbf{u}}=\|\mathbf{u}\|_{\lebe^{1}(\omega)}$ is a norm on $\ker(\mathbb{A};\Omega)$. Since $\Omega$ is assumed to be connected, by Lemma~\ref{lem:Celliptic}\ref{item:Cell2} $\ker(\mathbb{A};\Omega)$ consists of polynomials of a fixed maximal degree. Now, if a polynomial vanishes on a set of positive $\mathscr{L}^{n}$-measure, it vanishes globally. Hence, \eqref{eq:bulk} follows from Theorem \ref{thm:variant}. The asymptotics are clear, and the rest follows from Example \ref{ex:diffops}. This completes the proof. 
\end{proof}
\begin{remark}
If, in the situation of Proposition \ref{prop:bulk}, one drops the condition on $\Omega$ to be connected, Corollary \ref{cor:connectedness} implies that $\Omega$ has only finitely many connected components. Then  \eqref{eq:bulk} remains valid if and only if $\omega\subset\Omega$ satisfies $\mathscr{L}^{n}(\omega\cap\Omega_{j})>0$ for each of the connected components $\Omega_{j}$. 
\end{remark}
\subsection{Full boundary trace conditions}
We proceed with seminorms $\vertiii{\cdot}$ defined in terms of the \emph{full} boundary traces on subsets $\Gamma\subset\partial\Omega$. This is opposed to the \emph{partial} boundary traces from the next subsection, where we shall be interested in seminorms defined in terms of normal or tangential parts of the traces. The following result has first been observed by Pompe \cite{Pompe}:  
\begin{proposition}[Full boundary trace conditions]\label{prop:fullbdrytrace}
Let $\mathbb{A}$ be a first order $\mathbb{C}$-elliptic differential operator of the form \eqref{eq:form}. Moreover, let $\Omega\subset\R^{n}$ be open and bounded with Lipschitz boundary, and let $\Gamma\subset\partial\Omega$ be $\mathscr{H}^{n-1}$-measurable such that $\mathscr{H}^{n-1}(\Gamma\cap\partial\Omega_{j})>0$ holds for every connected component $\Omega_{j}$ of $\Omega$. Then, for every $1<p<\infty$,  there exists a constant $c>0$ such that 
\begin{align}\label{eq:fullbdry}
\|\mathbf{u}\|_{\sobo^{1,p}(\Omega)}\leq c\Big( \|\mathrm{tr}_{\partial\Omega}(\mathbf{u})\|_{\lebe^{1}(\Gamma)}+\|\mathbb{A}\mathbf{u}\|_{\lebe^{p}(\Omega)}\Big)\qquad\text{for all}\;\mathbf{u}\in\sobo^{1,p}(\Omega;V),  
\end{align}
where $\mathrm{tr}_{\partial\Omega}\colon\sobo^{1,p}(\Omega;V)\to\lebe^{p}(\partial\Omega;V)$ denotes the usual boundary trace operator on $\sobo^{1,p}(\Omega;V)$. Moreover, writing $c=c(\Gamma)$, we have $c(\Gamma)\to\infty$ as $\mathscr{H}^{n-1}(\Gamma)\searrow 0$. In particular, if $k=1$, \eqref{eq:fullbdry} holds for $\mathbb{A}\in\{\nabla^{D},\varepsilon\}$ if $n\geq 2$ and $\mathbb{A}=\sg^{D}$ if $n\geq 3$. 
\end{proposition}
\begin{proof}
If $\Omega$ is open and bounded with Lipschitz boundary in the sense of Section \ref{sec:validityfirst}, then it has at most finitely many connected components $\Omega_{1},...,\Omega_{N}$. Hence, the validity of Korn's first inequality on $\Omega$ follows. Now let $\mathbf{u}\in\ker(\mathbb{A};\Omega)$ be arbitrary. On each $\Omega_{j}$, $\mathbf{u}$ is the restriction of some $\mathbf{v}\in\ker(\mathbb{A};\R^{n})$. As established by Pompe \cite{Pompe}, $\|\mathbf{v}\|_{\lebe^{1}(\Gamma\cap\partial\Omega_{j})}=0$ implies that $\mathbf{v}\equiv 0$. In conclusion $\vertiii{\mathbf{u}}\coloneqq \|\mathrm{tr}_{\partial\Omega}(\mathbf{v})\|_{\lebe^{1}(\Gamma)}$ is a norm on $\ker(\mathbb{A};\Omega)$, whereby \eqref{eq:fullbdry} follows from Theorem \ref{thm:variant}. The proof is complete. 
\end{proof}

The corresponding trace inequality is substantially easier to obtain if $\Gamma=\partial\Omega$ and holds for a larger class of operators $\mathbb{A}$: 

\begin{remark}[Homogeneous versus inhomogeneous inequalities]\label{rem:hominhom}
If $\Gamma=\partial\Omega$ and $\bm{\rho}\in\ker(\mathbb{A};\Omega)$, then clearly $\bm{\rho}\in\sobo^{1,2}(\Omega;V)$. Now, if $\bm{\rho}|_{\partial\Omega}=0$, then $\bm{\rho}\in\sobo_{0}^{1,2}(\Omega;V)$ and thus the Korn inequality 
\begin{align}\label{eq:blackno1}
\int_{\Omega}|\nabla\mathbf{u}|^{2}\dif x\lesssim \int_{\Omega}|\mathbb{A}\mathbf{u}|^{2}\dif x \qquad\text{for all}\;\mathbf{u}\in\sobo_{0}^{1,2}(\Omega;V) 
\end{align}
implies that $\bm{\rho}=\mathrm{const.}$ and thus $\bm{\rho}=0$ because of $\bm{\rho}\in\sobo_{0}^{1,2}(\Omega;V)$.

We take this remark as an opportunity to comment on weaker ellipticity scenarios. We call a first order operator $\mathbb{A}$ of the form \eqref{eq:form} \emph{$\R$-elliptic} or simply \emph{elliptic} if $\mathbb{A}[\xi]\colon V \to W$ is injective for all $\xi\in\R^{n}\setminus\{0\}$. In this situation, inequality \eqref{eq:blackno1} holds true; more precisely, even for $1<p<\infty$ we have that 
\begin{align}\label{eq:lordpetrussteele}
\|\nabla\mathbf{u}\|_{\lebe^{p}(\Omega)}\lesssim\|\mathbb{A}\mathbf{u}\|_{\lebe^{p}(\Omega)}\qquad\text{for all}\;\mathbf{u}\in\sobo_{0}^{1,p}(\Omega;V).
\end{align}
Now, if $\mathbf{u}\in\sobo^{1,p}(\Omega;V)$, its trace $\mathrm{tr}_{\partial\Omega}(\mathbf{u})$ belongs to $\sobo^{1-1/p,p}(\partial\Omega;V)$. Then, one can find a function $\mathbf{g}\in\sobo^{1,p}(\Omega;V)$ such that 
\begin{align}\label{eq:lovin}
\|\mathbf{g}\|_{\sobo^{1,p}(\Omega)}\lesssim \|\mathrm{tr}_{\partial\Omega}(\mathbf{u})\|_{\sobo^{1-1/p,p}(\partial\Omega)}, 
\end{align}
and the right-hand side of the previous inequality cannot be improved. Since $(\mathbf{u}-\mathbf{g})\in\sobo_{0}^{1,p}(\Omega;V)$, we obtain 
\begin{align}\label{eq:cake}
\begin{split}
\|\nabla\mathbf{u}\|_{\lebe^{p}(\Omega)} & \lesssim \|\nabla(\mathbf{u}-\mathbf{g})\|_{\lebe^{p}(\Omega)} + \|\nabla \mathbf{g}\|_{\lebe^{p}(\Omega)}  \stackrel{\eqref{eq:lordpetrussteele}}{\lesssim} \|\mathbb{A}(\mathbf{u}-\mathbf{g})\|_{\lebe^{p}(\Omega)} + \|\nabla \mathbf{g}\|_{\lebe^{p}(\Omega)} \\ & \lesssim \|\mathbb{A}\mathbf{u}\|_{\lebe^{p}(\Omega)} + \|\nabla\mathbf{g}\|_{\lebe^{p}(\Omega)} \stackrel{\eqref{eq:lovin}}{\lesssim} \|\mathrm{tr}_{\partial\Omega}(\mathbf{u})\|_{\sobo^{1-1/p,p}(\partial\Omega)}+\|\mathbb{A}\mathbf{u}\|_{\lebe^{p}(\Omega)}. 
\end{split}
\end{align}
Since \eqref{eq:lovin} is sharp in general, meaning that its right-hand side cannot be replaced by any strictly weaker norm, this approach only yields a Korn-type inequality involving the $\sobo^{1-1/p,p}$-norm of the boundary traces \emph{but nothing weaker}. Yet, in the context of non-$\mathbb{C}$-elliptic operators, \eqref{eq:cake} is not useful in general. To underline this point, note that we have $\sobo^{1,p}(\Omega;V)=\sobo^{\mathbb{A},p}(\Omega)$ for $\mathbb{C}$-elliptic operators. Thereby \eqref{eq:fullbdry} tells us that, if we control $\mathrm{tr}_{\partial\Omega}(\mathbf{u})$ and $\mathbb{A}\mathbf{u}$, then we control the full Sobolev norm. Now, if $\mathbb{A}$ is not $\mathbb{C}$-elliptic, then there is \emph{no} trace operator $\mathrm{tr}_{\partial\Omega}\colon\sobo^{\mathbb{A},p}(\Omega)\to\sobo^{1-1/p,p}(\partial\Omega;V)$ (not even $\lebe_{\locc}^{1}(\partial\Omega;V)$), see, e.g., \cite[Thm. 4.4]{DieningGmeineder}. In particular, for $\sobo^{\mathbb{A},p}$-maps, the right-hand side of \eqref{eq:cake} is ill-defined. Put differently, there exists a sequence $(\mathbf{u}_{j})\subset\hold^{\infty}(\overline{\Omega};V)$ such that $\sup_{j\in\mathbb{N}}\|\mathbf{u}_{j}\|_{\sobo^{\mathbb{A},p}(\Omega)}<\infty$, but both sides of \eqref{eq:cake} blow up as $j\to\infty$. 
\end{remark}
\begin{remark}
Comparing \eqref{eq:bulk} and \eqref{eq:fullbdry}, the latter does not remain valid in general in the $k$-th order case if the left-hand side of \eqref{eq:fullbdry} is replaced by $\|\mathbf{u}\|_{\sobo^{k,p}(\Omega)}$. Indeed, let us consider $k=3$, $V=\R^{n}$, $\Omega=\ball_{1}(0)$. For $\mathbb{A}=\nabla^{3}$, we  consider $\mathbf{u}(x)\coloneqq (1-|x|^{2})e_{1}$. Then $\mathbf{u}|_{\partial\Omega}=0$ and $\mathbb{A}\mathbf{u}=0$, but the left-hand side of \eqref{eq:fullbdry} does not vanish. Yet, it would be interesting to characterise the algebraic variety $\Sigma_{\pi}\coloneqq \pi^{-1}(\{0\})$ for $\pi\in\ker(\mathbb{A};\R^{n})$ for general $\mathbb{C}$-elliptic operators $\mathbb{A}$. As Pompe's argument shows, at least in the first order case such varieties $\Sigma_{\pi}$ are highly rigid.
\end{remark}
A version of Proposition \ref{prop:fullbdrytrace} for $p=1$ also follows. It, in turn,  generalising a result for $\Gamma=\partial\Omega$ obtained by Steinke \cite{Steinke} who uses an elegant argument of different nature.
\begin{proposition}[Full boundary trace conditions, $p=1$]\label{prop:fullp=1} Retaining the same setting as in Proposition \ref{prop:fullbdrytrace}, there exists a constant $c>0$ such that 
\begin{align}\label{eq:p=1traces}
\|\mathbf{u}\|_{\lebe^{\frac{n}{n-1}}(\Omega)}\leq c\Big(\|\mathrm{tr}_{\partial\Omega}(\mathbf{u})\|_{\lebe^{1}(\Gamma)}+|\mathbb{A}\mathbf{u}|(\Omega)\Big)\qquad\text{for all}\;\mathbf{u}\in\mathrm{BV}^{\mathbb{A}}(\Omega). 
\end{align}

\end{proposition}
\begin{proof}
Since $\Omega$ is open and bounded with Lipschitz boundary, \eqref{eq:Sobolevp=1} is satisfied; see \cite{DieningGmeineder,GmeinederRaita}. As in the proof of Proposition \ref{prop:fullbdrytrace}, $\vertiii{\mathbf{u}}\coloneqq \|\mathrm{tr}_{\partial\Omega}(\mathbf{v})\|_{\lebe^{1}(\Gamma)}$ is a norm on $\ker(\mathbb{A};\Omega)$, and so \eqref{eq:p=1traces} follows from Proposition \ref{prop:1}. The proof is complete. 
\end{proof}
\subsection{Partial trace conditions}\label{sec:partialboundarytraces}
Different from the previous subsection, we now consider seminorms defined in terms of certain \emph{parts} of the full traces. Korn inequalities with such terms are required when dealing with problems which a priori only give us control over tangential or normal traces; as mentioned in the introduction, the resulting inequalities prove useful in the study of the Navier-Stokes equations with dynamic boundary conditions~\cite{GazcaEtAl,Verfuerth} or the Boltzmann equation~\cite{DesvilletesVillani}. For an open and bounded domain with $\hold^{1}$-boundary oriented by the outer unit normal $\nu_{\partial\Omega}\colon\partial\Omega\to\mathbb{S}^{n-1}$, we put for $\mathbf{u}\in\sobo^{1,p}(\Omega;\R^{n})$
\begin{align*}
& \mathrm{tr}_{\tau,\partial\Omega}(\mathbf{u}) \coloneqq (\mathrm{tr}_{\partial\Omega}(\mathbf{u})\cdot\nu_{\partial\Omega})\nu_{\partial\Omega}&\qquad\text{(normal  trace)}, \\ 
& \mathrm{tr}_{\nu,\partial\Omega}(\mathbf{u}) \coloneqq \mathrm{tr}_{\partial\Omega}(\mathbf{u}) - (\mathrm{tr}_{\partial\Omega}(\mathbf{u})\cdot\nu_{\partial\Omega})\nu_{\partial\Omega}&\qquad\text{(tangential trace).}
\end{align*}
We begin with the deviatoric gradient $\mathbb{A}=\nabla^{D}$, and recall that its kernel is   $\ker(\nabla^D) = \{x \mapsto \alpha x + b \colon \alpha \in \mathbb{R}, b \in \mathbb{R}^n\}$, see Example~\ref{ex:diffops}. Let us start with the following elementary lemma. 

\begin{lemma}[A characterisation of balls]\label{lem:devgradtan}
Let $\Omega\subset\R^{n}$ be an open, bounded and connected set with $\hold^{1}$-boundary. Then the following are equivalent: 
\begin{enumerate}
	\item\label{item:charball1} $\Omega$ is a ball. 
	\item\label{item:charball2} There exist {$\alpha \in\R \backslash \{0\}$} and $b\in\R^{n}$, not both equal to zero, are such that $\nu_{\partial\Omega}(x)$ is \emph{collinear} to $(\alpha x+b)$ for all $x\in\partial\Omega$. 
\end{enumerate}
\end{lemma}
\begin{proof}
On '\ref{item:charball1}$\Rightarrow$\ref{item:charball2}'. Suppose that $\Omega=\ball_{r}(x_{0})$ for some $x_{0}\in\R^{n}$ and $r>0$. For an arbitrary $\alpha\in\R$, put $b\coloneqq- \alpha x_{0}$. Then $\bm{\rho}(x)=\alpha x+b = \alpha(x- x_0)$ is collinear to 
\begin{align*}
\nu_{\partial\Omega}(x) = 
\frac{x-x_{0}}{|x-x_{0}|}
\qquad\text{for all}\;x\in\partial\Omega, 
\end{align*}
and hence~\ref{item:charball2} follows. 

On '\ref{item:charball2}$\Rightarrow$\ref{item:charball1}'.
Now suppose that $\alpha \in\R \backslash \{0\}$ and $b\in\R^{n}$ are such that \ref{item:charball2} holds. For arbitrary $x,y\in\partial\Omega$ with $x\neq y$, we let $
\gamma_{x,y} 
\colon[0,1]\to\partial\Omega$ be a continuously differentiable curve with $\gamma_{x,y}(0)=x$ and $\gamma_{x,y}(1)=y$; this is precisely where we use that $\partial\Omega$ is of class $\hold^{1}$ and that $\Omega$ is connected. By the collinearity assumption from \ref{item:charball2}, there is a function $f \colon [0,1] \to \{-1,1\}$ such that 
\begin{align}\label{eq:nugamma}
\nu_{\partial\Omega}(\gamma_{x,y}(r))=f(r)\frac{\alpha \gamma_{x,y}(r)+b}{|\alpha \gamma_{x,y}(r)+b|}
\qquad\text{for all}\;0\leq r \leq 1, 
\end{align}
On the other hand, $\nu_{\partial\Omega}(\gamma_{x,y}(r))\bot\gamma'_{x,y}(r)$ for all $0\leq r \leq 1$, and hence  \eqref{eq:nugamma} yields 
\begin{align}\label{eq:spherecomp2}
0=( \alpha \gamma_{x,y}(r)+b)\cdot\gamma'_{x,y}(r) 
= \alpha \gamma_{x,y}(r)\cdot\gamma'_{x,y}(r) + b\cdot\gamma'_{x,y}(r)\qquad\text{for all}\;0\leq r \leq 1. 
\end{align}
In consequence, the fundamental theorem of calculus gives us 
\begin{align}\label{eq:spherecomp}
	\begin{split}
\frac{\alpha}{2}(|x|^{2}-|y|^{2}) & = \frac{\alpha}{2}(|\gamma_{x,y}(1)|^{2}-|\gamma_{x,y}(0)|^{2}) \\ & = \frac{\alpha}{2}\int_{0}^{1}\frac{\dif}{\dif r}|\gamma_{x,y}(r)|^{2}\dif r  = \int_{0}^{1}\alpha \gamma'_{x,y}(r)\cdot\gamma_{x,y}(r)\dif r \\ 
& \!\!\! \stackrel{\eqref{eq:spherecomp2}}{=} - b\cdot \int_{0}^{1}\gamma'_{x,y}(r)\dif r = -b\cdot (\gamma_{x,y}(1)-\gamma_{x,y}(0)) = -b\cdot(x-y). 
\end{split}
\end{align}
Completing the square, \eqref{eq:spherecomp} implies that 
\begin{align*}
\left\vert x + \frac{2}{\alpha}b\right\vert^2 = \left\vert y + \frac{2}{\alpha}b\right\vert^2 \qquad\text{for all}\;x,y\in\partial\Omega\;\text{with}\;x\neq y. 
\end{align*}
This shows that $\partial\Omega$ is a sphere centred at $\frac{2}{\alpha}b$ and so \ref{item:charball1} follows. The proof is complete. 
\end{proof}
The previous proof formalises the fact that, for a field $\bm{\rho}(x)=ax+b$, it is impossible to be globally orthogonal to the tangents $\tau_{\partial\Omega}$ along $\partial\Omega$ unless $\partial\Omega$ is a sphere; see Figure \ref{fig:sphere}(i). Here, an instance of a set and boundary point $y_{0}$ is depicted where this condition fails. We now have: 
\begin{proposition}[Partial traces, deviatoric gradient]\label{prop:devgradpartial}
Let $1<p<\infty$, and let  $\Omega\subset\R^{n}$ be open, bounded with $\hold^{1}$-boundary. Then the following hold: 
\begin{enumerate}
\item\label{item:partial-devgrad-tan} The inequality 
	\begin{align}\label{eq:tanbdrydevgrad}
		\|\mathbf{u}\|_{\sobo^{1,p}(\Omega)}\lesssim \int_{\partial\Omega}|\mathrm{tr}_{\tau,\partial\Omega}(\mathbf{u})|\dif\mathscr{H}^{n-1} + \|\nabla^{D}\mathbf{u}\|_{\lebe^{p}(\Omega)}\qquad\text{for all}\;\mathbf{u}\in\sobo^{1,p}(\Omega;\R^{n})
	\end{align}
	holds if and only no connected component of $\Omega$ is a ball. In particular, if $\Omega$ is connected, then the validity of \eqref{eq:tanbdrydevgrad} is equivalent to $\Omega$ not being a ball.  
\item\label{item:partial-devgrad-nu}  The inequality 
\begin{align}\label{eq:normbdrydevgrad}
		\|\mathbf{u}\|_{\sobo^{1,p}(\Omega)}\lesssim \int_{\partial\Omega}|\mathrm{tr}_{\nu,\partial\Omega}(\mathbf{u})|\dif\mathscr{H}^{n-1} + \|\nabla^{D}\mathbf{u}\|_{\lebe^{p}(\Omega)}\qquad\text{for all}\;\mathbf{u}\in\sobo^{1,p}(\Omega;\R^{n})
	\end{align}
    holds without further restrictions on $\Omega$. 
\end{enumerate}
\end{proposition}
\begin{proof}
Our {definition} of open and bounded sets with $\hold^{1}$-boundaries implies that $\Omega$ has at most finitely many connected components $\Omega_{1},...,\Omega_{N}$. In particular, \eqref{eq:KornJohn} in conjunction with Remark~\ref{rem:relations} implies the validity of the first Korn inequality for $\mathbb{A}=\nabla^{D}$ on such sets. 

\begin{figure}[t]
\begin{tikzpicture}
\node at (0,0) {\textbullet};
\node[right] at (0,-0.25) {$x_{0}$};
\draw[->,dashed] (0,2) -- (-2,2);
\draw[->,dashed,rotate=10] (0,2) -- (-2,2);
\draw[->,dashed,rotate=20] (0,2) -- (-2,2);
\draw[->,dashed,rotate=30] (0,2) -- (-2,2);
\draw[->,dashed,rotate=40] (0,2) -- (-2,2);
\draw[->,dashed,rotate=50] (0,2) -- (-2,2);
\draw[->,dashed,rotate=60] (0,2) -- (-2,2);
\draw[->,dashed,rotate=70] (0,2) -- (-2,2);
\draw[->,dashed,rotate=80] (0,2) -- (-2,2);
\draw[->,dashed,rotate=90] (0,2) -- (-2,2);
\node[rotate=-30] at (-1.35,-2) {$.....$};
\draw[<->,dotted] (0,-2) -- (0,2);
\draw[<->,dotted,rotate=10] (0,-2) -- (0,2);
\draw[<->,dotted,rotate=20] (0,-2) -- (0,2);
\draw[<->,dotted,rotate=30] (0,-2) -- (0,2);
\draw[<->,dotted,rotate=40] (0,-2) -- (0,2);
\draw[<->,dotted,rotate=50] (0,-2) -- (0,2);
\draw[<->,dotted,rotate=60] (0,-2) -- (0,2);
\draw[<->,dotted,rotate=70] (0,-2) -- (0,2);
\draw[<->,dotted,rotate=80] (0,-2) -- (0,2);
\draw[<->,dotted,rotate=90] (0,-2) -- (0,2);
\draw[<->,dotted,rotate=100] (0,-2) -- (0,2);
\draw[<->,dotted,rotate=110] (0,-2) -- (0,2);
\draw[<->,dotted,rotate=120] (0,-2) -- (0,2);
\draw[<->,dotted,rotate=130] (0,-2) -- (0,2);
\draw[<->,dotted,rotate=140] (0,-2) -- (0,2);
\draw[<->,dotted,rotate=150] (0,-2) -- (0,2);
\draw[<->,dotted,rotate=160] (0,-2) -- (0,2);
\draw[<->,dotted,rotate=170] (0,-2) -- (0,2);
\draw[<->,dotted,rotate=180] (0,-2) -- (0,2);
\node at (-1.1,2.3) {$\tau_{\partial \Omega}$};
\draw (0,0) circle (2cm);
\draw[-,rounded corners=3mm,fill=black!30!white,opacity=0.2] (-1,-1) -- (0,-0.5) -- (1,-1) -- (0.5,0) -- (1,1) -- (0,0.5) -- (-1,1) -- (-0.5,0) [out=-110, in = 180] to (-1,-1);
\node at (-0.45,0.725) {{\tiny\textbullet}};
\node[below] at (-0.45,0.725) {{$y_{0}$}};
\draw[dashed,->]  (-0.45,0.725) --(-1.3,1.2);
\node at (0,-2.5) {(i)};
\end{tikzpicture}
\hspace{1cm}
\begin{tikzpicture}
\node at (0,0) {\textbullet};
\node[right] at (0,-0.25) {$x_{0}$};
\draw[<->,dotted] (0,-2) -- (0,2);
\draw[<->,dotted,rotate=10] (0,-2) -- (0,2);
\draw[<->,dotted,rotate=20] (0,-2) -- (0,2);
\draw[<->,dotted,rotate=30] (0,-2) -- (0,2);
\draw[<->,dotted,rotate=40] (0,-2) -- (0,2);
\draw[<->,dotted,rotate=50] (0,-2) -- (0,2);
\draw[<->,dotted,rotate=60] (0,-2) -- (0,2);
\draw[<->,dotted,rotate=70] (0,-2) -- (0,2);
\draw[<->,dotted,rotate=80] (0,-2) -- (0,2);
\draw[<->,dotted,rotate=90] (0,-2) -- (0,2);
\draw[<->,dotted,rotate=100] (0,-2) -- (0,2);
\draw[<->,dotted,rotate=110] (0,-2) -- (0,2);
\draw[<->,dotted,rotate=120] (0,-2) -- (0,2);
\draw[<->,dotted,rotate=130] (0,-2) -- (0,2);
\draw[<->,dotted,rotate=140] (0,-2) -- (0,2);
\draw[<->,dotted,rotate=150] (0,-2) -- (0,2);
\draw[<->,dotted,rotate=160] (0,-2) -- (0,2);
\draw[<->,dotted,rotate=170] (0,-2) -- (0,2);
\draw[<->,dotted,rotate=180] (0,-2) -- (0,2);
\draw[black!30!white,fill=black!30!white,opacity=0.2] (0,0) -- (2.5,2.1115) -- (2.5,-2.1115) -- (0,0);
\node at (2.5,2.1115) {\tiny{\textbullet}};
\node[right] at (2.5,2.1115) {$z_{0}$};
\node at (2.5,2.1115) {\tiny{\textbullet}};
\node at (2.5,-2.1115) {\tiny{\textbullet}};
\node[right] at (2.5,-2.1115) {$z'_{0}$};
\node at (0,-2.5) {(ii)};
\end{tikzpicture}
\caption{Critical geometric scenarios for the deviatoric gradient; the finely dotted arrows indicating the vector field of kernel element $\alpha x+b$.}
\label{fig:sphere}
\end{figure}
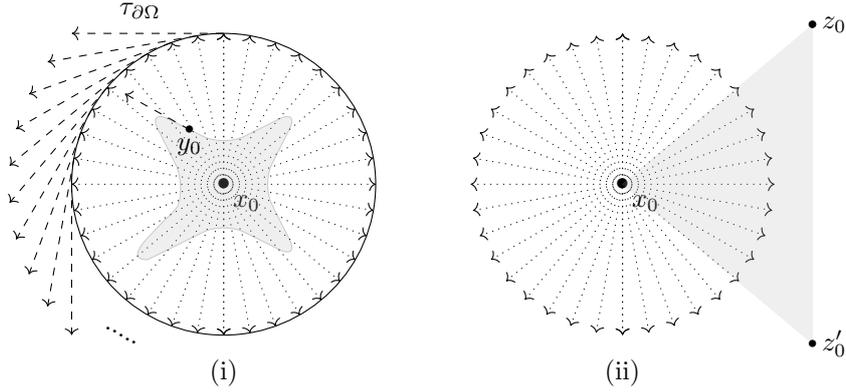

On \ref{item:partial-devgrad-tan}.  Suppose that \eqref{eq:tanbdrydevgrad} holds, and suppose that there exists $j_{0}\in\{1,...,N\}$ such that $\Omega_{j_{0}}$ is a ball: $\Omega_{j_{0}}=\ball_{r}(x_{0})$. By Lemma \ref{lem:devgradtan}, there exist $a\in\R$ and $b\in\R^{n}$, not both equal to zero, such that $\nu_{\partial\Omega_{j_{0}}}(x)$ is collinear to $\bm{\rho}(x)= \alpha x+b$ for all $x\in\partial\Omega_{j_{0}}$. We put 
\begin{align*}
\widetilde{\bm{\rho}}(x)\coloneqq\mathbbm{1}_{\ball_{r}(x_{0})}\bm{\rho}(x),\qquad x\in\Omega. 
\end{align*}
Setting $\mathbf{u}\coloneqq\widetilde{\bm{\rho}}$, then the right-hand side of \eqref{eq:tanbdrydevgrad} vanishes, whereas the left-hand side does not. This is a contradiction, and hence no connected component of $\Omega$ can be a ball. Conversely, suppose that no connected component of $\Omega$ is a ball. Lemma \ref{lem:devgradtan} then tells us that  
\begin{align}\label{eq:devgradproofseminorm}
\vertiii{\mathbf{u}}\coloneqq\int_{\partial\Omega_{j}}|\mathrm{tr}_{\tau,\partial\Omega_{j}}(\mathbf{u})|\dif\mathscr{H}^{n-1} 
\end{align}
is a norm on $\ker(\nabla^{D};\Omega_{j})$ for every $j\in\{1,...,N\}$. Hence, Theorem \ref{thm:variant} yields that \eqref{eq:tanbdrydevgrad} (with $\Omega$ being replaced by $\Omega_{j}$) holds for every $j\in\{1,...,N\}$. Adding the resulting inequalities then gives us \eqref{eq:tanbdrydevgrad} globally on $\Omega$, which is \ref{item:partial-devgrad-tan}.

On \ref{item:partial-devgrad-nu} . We consider the seminorm $\vertiii{\cdot}$ as in \eqref{eq:devgradproofseminorm}, where now $\partial\Omega_{j}$ is replaced by $\partial\Omega$ and $\mathrm{tr}_{\tau,\partial\Omega_{j}}$ is replaced by $\mathrm{tr}_{\nu,\partial\Omega}$. Then $\vertiii{\cdot}$ is a norm on the space of maps $\bm{\rho}(x)= \alpha x+b$, for $\alpha \in \mathbb{R}$ and $b \in \mathbb{R}^n$. To see this, we confine ourselves to a sketch. If there exists such a non-trivial map $\bm{\rho}$ such that $\vertiii{\bm{\rho}}=0$, then the boundary must contain at least one line of the form $x_{0}+\R_{\geq 0}{w}_0$ with some $x_{0}\in\partial\Omega$ and ${w}_0 \in\R^{n}\setminus\{0\}$. If it contains one  such line, then $\Omega$ cannot be bounded, see Figure \ref{fig:sphere}(ii). 
Hence, $\vertiii{\cdot}$ is a seminorm, and so \eqref{eq:normbdrydevgrad} follows from Theorem \ref{thm:variant}. The proof is complete. 
\end{proof}
We expect the 'if-and-only-if'-result of the preceding proposition to hold for a class of more irregular domains.

Next, let us turn to the symmetric gradient with kernel $ \ker(\varepsilon) = \mathscr{R}(\R^{n}) = \{x \mapsto Ax+b\colon\;A\in\R_{\mathrm{skew}}^{n\times n},\,b\in\R^{n}\} $, see Example~\ref{ex:diffops}\ref{item:symgrad}. 
The following proposition extends 
results in~\cite{DesvilletesVillani} on homogeneous normal trace conditions. 

\begin{example}[Partial traces, symmetric gradient]\label{prop:partialtracesymgrad}
Let $1<p<\infty$, and let  $\Omega\subset\R^{n}$ be open, bounded with $\hold^{1}$-boundary. Then the following hold: 
\begin{enumerate}
\item\label{item:partial-symgrad-tan} The inequality 
	\begin{align}\label{eq:tanbdrysymgrad}
		\|\mathbf{u}\|_{\sobo^{1,p}(\Omega)}\lesssim \int_{\partial\Omega}|\mathrm{tr}_{\tau,\partial\Omega}(\mathbf{u})|\dif\mathscr{H}^{n-1} + \|\sg (\mathbf{u})\|_{\lebe^{p}(\Omega)}\qquad\text{for all}\;\mathbf{u}\in\sobo^{1,p}(\Omega;\R^{n})
	\end{align}
	holds without further restrictions on $\Omega$. 
\item\label{item:partial-symgrad-nu}  The inequality 
\begin{align}\label{eq:normbdrysymgrad}
		\|\mathbf{u}\|_{\sobo^{1,p}(\Omega)}\lesssim \int_{\partial\Omega}|\mathrm{tr}_{\nu,\partial\Omega}(\mathbf{u})|\dif\mathscr{H}^{n-1} + \|\sg (\mathbf{u})\|_{\lebe^{p}(\Omega)}\qquad\text{for all}\;\mathbf{u}\in\sobo^{1,p}(\Omega;\R^{n})
	\end{align}
    holds provided that $\Omega$ is not axisymmetric. 
\end{enumerate}
     To show~\ref{item:partial-symgrad-tan}, a similar argument as in~Proposition~\ref{prop:devgradpartial}\ref{item:partial-devgrad-nu} can be applied.  
     Indeed,  any $\hold^{1}$-domain whose boundary tangent vectors are orthogonal to rigid body rotations is unbounded. 
     
     With the arguments used in \cite{DesvilletesVillani} one can show that $\vertiii{\bu} = \int_{\partial\Omega}|\mathrm{tr}_{\nu,\partial\Omega}(\mathbf{u})|\dif\mathscr{H}^{n-1}$ is a norm on $\ker(\sg)$ if and only if $\Omega$ is not axisymmetric. This proves \ref{item:partial-symgrad-nu}. 
\end{example}

\begin{example}[Deviatoric symmetric gradients]\label{ex:devsymgradcomplications}
Determining the domains such that 
normal trace or tangential trace terms are a norm on $\ker(\sg^D)$ is significantly harder, recall Example~\ref{ex:diffops}\ref{item:devsymgrad}. 
Clearly, for those domains, for which those terms are not a norm on the kernel $\sg$, are also excluded here, since by Example~\ref{ex:diffops} we have 
\begin{align*}
         \ker(\sg)
 \subset  \ker(\sg^D).
\end{align*}
Beyond this, however, we cannot assert much: While the kernel of $\sg^D$ is a subspace of quadratic functions, the condition that a kernel element satisfies zero normal or zero tangential trace conditions can be reformulated as a PDE; going from full trace to partial trace conditions thus comes with a switch from an  algebraic problem to an analytic one.
\end{example}

\begin{remark}[Partial traces on part of the boundary]
Similar to Proposition~\ref{prop:fullbdrytrace}, it is also of interest to impose partial boundary conditions only on a part of the boundary $\Gamma \subset \partial \Omega$. 
It is evident that the class of domains for which the corresponding seminorm $\vertiii{\cdot}$ is also a norm on the kernel of the respective differential operator is possibly larger than for $\Gamma = \partial \Omega$. 
\end{remark}
Based on Example \ref{ex:devsymgradcomplications}, it is challenging if not impossible to obtain explicit characterisations of all  available couplings between the partial boundary conditions, differential operators and the geometry of domains to yield the corresponding Korn-type inequalities. For readers interested in this matter, Section \ref{sec:sym-comput} provides a symbolic test to decide whether a seminorm can qualify as a norm on the nullspace of $\mathbb{A}$ for given  partial trace conditions. 
\subsection{Other space scales}\label{sec:otherscales}
For problems from material science with, e.g., logarithmic hardening \cite{FuchsSeregin} or limiting cases in fluid mechanics such as the Prandtl-Eyring model \cite{BreitDieningFuchs},  $\lebe^{p}$-based Korn inequalities do not suffice. Korn inequalities of Orlicz-type have been obtained in \cite{BreitCianchiDiening,BreitDiening,Cianchi,DieningRuzickaSchumacher,FuchsOrlicz} for specific operators such as $\sg$ or $\sg^{D}$, and in \cite{ContiGmeineder,DieningGmeineder,Stephan} in the general case; most of them crucially hinge on Cianchi's sharp singular integral estimates \cite{Cianchi1}. 

To state them, we recall some terminology. A function $\varphi\colon [0,\infty)\to[0,\infty]$ is called a \emph{Young function} if it can be written as
\begin{align*}
\varphi(t) = \int_{0}^{t}\psi(s)\dif s,\qquad t\geq 0, 
\end{align*}
where $\psi\colon[0,\infty)\to[0,\infty]$ is non-decreasing, left-continuous and neither identical to $0$ nor to $+\infty$. Its \emph{convex conjugate} is given by $\varphi^{*}(t)\coloneqq\sup_{s>0}st-\varphi(s)$. A Young function gives rise to the Luxemburg norm 
\begin{align*}
\|u\|_{\lebe^{\varphi}(\Omega)} \coloneqq \inf\Big\{\lambda>0\colon\;\int_{\Omega}\varphi\Big(\frac{|u|}{\lambda}\Big)\dif x \leq 1 \Big\},  
\end{align*}
and the Orlicz space $\lebe^{\varphi}(\Omega)$ is defined as the collection of all measurable $u\colon\Omega\to\R$ such that $\|u\|_{\lebe^{\varphi}(\Omega)}<\infty$. In particular, $\lebe^{\varphi}(\Omega)$ is a Banach function space. For two Young functions $\varphi,\widetilde{\varphi}$, we shall consider the following balance condition: There exist $c>0$ and $t_{0}\geq 0$ such that 
\begin{align}\label{eq:Orliczbalance}
t\int_{t_{0}}^{t}\frac{\varphi(s)}{s^{2}}\dif s<\widetilde{\varphi}(ct)\;\;\;\text{and}\;\;\;t\int_{t_{0}}^{t}\frac{\widetilde{\varphi}^{*}(s)}{s^{2}}\dif s \leq \varphi^{*}(ct)\qquad\text{for all}\;t\geq t_{0}. 
\end{align}
\begin{example}
As mentioned above, plasticity models with logarithmic hardening or the Prandtl-Eyring model from fluid mechanics are based on the space $\mathrm{L}\log\mathrm{L}(\Omega)$, which corresponds to the Young function $\widetilde{\varphi}(t)\coloneqq t\log(t+1)$. Then $\widetilde{\varphi}^{*}$ is of exponential growth. We moreover put $\varphi(t)\coloneqq t$, so that
\begin{align*}
\varphi^{*}(t)= \begin{cases}
0&\;0\leq t \leq 1,\\ 
+\infty&\;1<t<\infty. 
\end{cases}
\end{align*}
For these choices, a straightforward computation shows that \eqref{eq:Orliczbalance} is satisfied. 
\end{example}
Now let $\mathbb{A}$ be a first order $\mathbb{C}$-elliptic differential operator of the form \eqref{eq:form} and let  $\varphi,\widetilde{\varphi}$ be of class $\Delta_{2}$; this means that there exists $c>0$ such that $\varphi(2t)\leq c\varphi(t)$ holds for all $t\geq 0$, and analogously for $\widetilde{\varphi}$. Subject to the above balance condition \eqref{eq:Orliczbalance}, we then have the Korn-type inequality 
\begin{align}\label{eq:DerKorner}
\|\mathbf{u}\|_{\sobo^{1}\lebe^{\varphi}(\Omega)}\lesssim\|\mathbf{u}\|_{\lebe^{\widetilde{\varphi}}(\Omega)}+\|\mathbb{A}\mathbf{u}\|_{\lebe^{\widetilde{\varphi}}(\Omega)}\qquad\text{for all}\;\mathbf{u}\in\sobo^{1}\lebe^{\varphi}(\Omega;V),
\end{align}
provided that $\Omega$ is an open and bounded  $\sobo^{\mathbb{A}}\lebe^{\widetilde{\varphi}}$-extension domain; this follows from the results and techniques of \cite{Cianchi,Cianchi1} and \cite{GmeinederRaita}. In particular, this is satisfied if $\Omega$ has Lipschitz boundary. 
\begin{proposition}[Orlicz spaces]\label{prop:Orlicz}
Let $\Omega\subset\R^{n}$ be open, connected and bounded  with Lipschitz boundary, and suppose that the Young functions $\varphi,\widetilde{\varphi}$ are of class $\Delta_{2}$ and satisfy the balance condition \eqref{eq:Orliczbalance}. Moreover, let $\mathbb{A}$ be a first order $\mathbb{C}$-elliptic operator of the form \eqref{eq:form}. Then the following hold: 
\begin{enumerate}
\item\label{item:Orlicz1} For any $\mathscr{L}^{n}$-measurable set $\omega\subset\Omega$ with $\mathscr{L}^{n}(\omega)>0$, we have 
\begin{align*}
\|\mathbf{u}\|_{\sobo^{1}\lebe^{\varphi}(\Omega)} \lesssim \|\mathbf{u}\|_{\lebe^{1}(\omega)} + \|\mathbb{A}\mathbf{u}\|_{\lebe^{\widetilde{\varphi}}(\Omega)}\qquad\text{for all}\;\mathbf{u}\in\sobo^{1}\lebe^{\varphi}(\Omega;V).
\end{align*}
\item\label{item:Orlicz2} For any $\mathscr{H}^{n-1}$-measurable set $\Gamma\subset\partial\Omega$ with $\mathscr{H}^{n-1}(\Gamma)>0$, we have 
\begin{align*}
\|\mathbf{u}\|_{\sobo^{1}\lebe^{\varphi}(\Omega)} \lesssim \|\mathrm{tr}_{\partial\Omega}(\mathbf{u})\|_{\lebe^{1}(\Gamma)} + \|\mathbb{A}\mathbf{u}\|_{\lebe^{\widetilde{\varphi}}(\Omega)}\qquad\text{for all}\;\mathbf{u}\in\sobo^{1}\lebe^{\varphi}(\Omega;V).
\end{align*}
\end{enumerate}

\end{proposition}
\begin{proof}
We directly employ Theorem \ref{thm:korn-gen}; the underlying Poincar\'{e} inequality from $\eqref{eq:kornetto0}_{2}$  follows from \cite[Sec.~3]{DieningGmeineder}. In light of \eqref{eq:DerKorner}, we only have to verify that the respective seminorms are norms on the nullspace of $\mathbb{A}$, but this is the case by the proofs of Propositions~\ref{prop:bulk} and \ref{prop:fullbdrytrace}. 
This completes the proof. 
\end{proof}
With the obvious modifications, a similar result can also be obtained for partial traces. Note that, by the very same method of proof, it is also possible to establish Korn-type inequalities, e.g., for Lorentz spaces or Lebesgue spaces weighted with suitable $A_{p}$-Muckenhoupt weights. Indeed, this primarily requires Korn's first inequality for these scales, in turn to be found, e.g., in~\cite{DieningGmeineder,DieningRuzickaSchumacher}.

\section{Interior and lower dimensional trace conditions}\label{sec:choices}
In this section, we focus on specific instances of Korn-type inequalities where the seminorms $\vertiii{\cdot}$ are given by integrals of $\mu$- and so potentially lower dimensional traces of Sobolev functions. 
\subsection{$\mu$-traces}\label{sec:mutraces}
For the reader's convenience, we briefly recall restriction or trace theorems for $\sobo^{k,p}(\R^{n})$, for  $1<p<\infty$. Denoting by $\mathcal{G}_{k}\coloneqq\mathscr{F}^{-1}((1+|\xi|^{2})^{-\frac{k}{2}})$ the $k$-th order Bessel potential, we note that every $u\in\sobo^{k,p}(\R^{n})$ can be written as $u=\mathcal{G}_{k}*f$ with some $f\in\lebe^{p}(\R^{n})$. This is the Bessel potential characterisation of $\sobo^{k,p}(\R^{n})$, see \cite[Thm.~1.2.3]{AdamsHedberg}, and we have the norm equivalence 
\begin{align}\label{eq:normequivBesselSobolev}
\|u\|_{\sobo^{k,p}(\R^{n})}\eqsim \inf\{\|f\|_{\lebe^{p}(\R^{n})}\colon\;u=\mathcal{G}_{k}*f,\;f\in\lebe^{p}(\R^{n})\},\qquad u\in\sobo^{k,p}(\R^{n}). 
\end{align}
The $(k,p)$-\emph{Sobolev capacity} of a compact set $K\subset\R^{n}$ is given by 
\begin{align*}
\mathrm{Cap}_{k,p}(K)\coloneqq \inf\{\|\varphi\|_{\sobo^{k,p}(\R^{n})}^{p}\colon\;\varphi\in\hold_{c}^{\infty}(\R^{n})\;\text{and}\;\varphi\geq 1\;\text{on}\;K\}. 
\end{align*}
For an open set $U\subset\R^{n}$, this definition is extended via 
\begin{align*}
\mathrm{Cap}_{k,p}(U)\coloneqq \sup\{\mathrm{Cap}_{k,p}(K)\colon\;K\subset U\;\text{is compact}\}. 
\end{align*}
For general sets $V\subset \R^{n}$, one then defines an outer capacity by
\begin{align*}
\mathrm{Cap}_{k,p}(V)\coloneqq \inf\{\mathrm{Cap}_{k,p}(U)\colon\;V\subset U,\;U\;\text{is open}\}. 
\end{align*}
Now let $\mu\in\mathcal{M}^{+}(\R^{n})$. In the following, we assume that $1<p<q<\infty$,  and that there exists a constant $c>0$ such that 
\begin{align}\label{eq:capcondo1}
\mu(K)^{\frac{1}{q}}\leq c\,\mathrm{Cap}_{k,p}(K)^{\frac{1}{p}}\qquad \text{for all compact sets}\;K\subset\R^{n}.
\end{align}
By Adams'  trace inequality (see, e.g., \cite[Thm.~7.2.1]{AdamsHedberg}), \eqref{eq:capcondo1} implies that the $k$-th order Bessel potential operator $G_{k}\colon f \mapsto \mathcal{G}_{k}*f$ maps $G_{k}\colon\lebe^{p}(\R^{n};\mathscr{L}^{n})\to\lebe^{q}(\R^{n};\mu)$ boundedly. In particular, by \eqref{eq:normequivBesselSobolev}, there exists $c>0$ such that $\|u\|_{\lebe^{q}(\R^{n};\mu)}\leq c\|u\|_{\sobo^{k,p}(\R^{n})}$ holds for all $u\in\hold_{c}^{\infty}(\R^{n})$. 

Now, if $u\in\sobo^{k,p}(\R^{n})$, we choose $(u_{j})\subset\hold_{c}^{\infty}(\R^{n})$ such that $u_{j}\to u$ strongly in $\sobo^{k,p}(\R^{n})$. Since $(u_{j})$ is Cauchy with respect to $\|\cdot\|_{\sobo^{k,p}(\R^{n})}$, the preceding inequality implies that $(u_{j})$ is Cauchy in $\lebe^{q}(\R^{n};\mu)$. By the Banach space property of the latter, there exists an element $\mathrm{tr}_{\mu}(u)\in\lebe^{q}(\R^{n};\mu)$ such that $\|u_{j}-\mathrm{tr}_{\mu}(u)\|_{\lebe^{q}(\R^{n};\mu)}\to 0$. Moreover, $\mathrm{tr}_{\mu}(u)$ is independent of the particular approximating sequence, and hence gives rise to a well-defined, bounded linear \emph{$\mu$-trace operator} $
\mathrm{tr}_{\mu}\colon\sobo^{k,p}(\R^{n})\to\lebe^{q}(\R^{n};\mu)$. As examples, if 
\begin{itemize}
\item $\mu=\mathscr{H}^{n-1}\mres\Sigma$ with some sufficiently smooth hypersurface $\Sigma$, or  
\item $\mu=\mathscr{H}^{n-s}\mres\Sigma$ for some $0<s<n$ and a suitably regular $(n-s)$-dimensional set  (such as lines if $s=n-1$ or certain fractals if $0<s<n$), 
\end{itemize}
then $\mathrm{tr}_{\mu}(u)$ can be understood as a (generalised) restriction of $u$ to $\Sigma$. 
\subsection{Full gradient inequalities}
Even though setting $\mathbb{A}=\nabla$ does not give rise to a proper Korn-type inequality, we begin with the following base case of independent interest. 
\begin{example}[Gradient, $\mu$-traces]\label{ex:gradient} 
Let $1<p\leq q<\infty$, and let $\Omega\subset\R^{n}$ be an open, bounded and connected domain such that $\sobo^{1,p}(\Omega)\hookrightarrow\hookrightarrow\lebe^{p}(\Omega)$; e.g., $\Omega$ might be an extension domain for $\sobo^{1,p}$ or a John domain. Moreover, let $\mu\in\mathcal{M}^{+}(\R^{n})$ be compactly supported in $\Omega$ such that \eqref{eq:capcondo1} holds and  $\mu(K_{0})> 0$ for some compact subset $K_{0}\subset\Omega$. Then we have 
\begin{align}\label{eq:traceineq1}
\|u\|_{\sobo^{1,p}(\Omega)} \lesssim \Big(\int_{\Omega}|\mathrm{tr}_{\mu}(u)|\dif\mu + \|\nabla u\|_{\lebe^{p}(\Omega)}\Big)\qquad\text{for all}\;u\in\sobo^{1,p}(\Omega). 
\end{align}
In the framework of Theorem \ref{thm:variant}, we put $X=\lebe^{p}$, $k=1$, $\mathbb{A}=\nabla$ and 
\begin{align}\label{eq:chiefseminormSection4}
\vertiii{u}\coloneqq\int_{\Omega}|\mathrm{tr}_{\mu}(u)|\dif\mu,\qquad u\in\sobo^{1,p}(\Omega). 
\end{align}
By the connectedness of $\Omega$, any $\rho\in\ker(\nabla;\Omega)$ is constant and so $u\equiv d$ for some $d\in\R$. Hence, 
\begin{align*}
\vertiii{\rho}=0 \;\; \Longrightarrow \;\; 
|d|\mu(K_{0}) \leq  |d|\mu(\Omega) = \int_{\Omega}|\mathrm{tr}_{\mu}(\rho)|\dif\mu \stackrel{!}{=} 0\;\; \stackrel{\mu(K_{0})>0}{\Longrightarrow}\;\; u\equiv 0.
\end{align*}
In conclusion, \eqref{eq:traceineq1} follows from Theorem \ref{thm:variant}, see \eqref{eq:improvedKorn-v}. Note that, if $p>n$ and $\Omega$ has Lipschitz boundary, the sole restriction on $\mu$ (besides being compactly supported inside $\Omega$) is that $\mu(K_{0})>0$ for some compact subset $K_{0}\subset\Omega$. This is due to the fact that $\sobo^{1,p}(\Omega)\hookrightarrow\hold(\overline{\Omega})$ in this case. 
\end{example}

\subsection{Korn-type inequalities involving $\mu$-traces} 
In analogy to Section \ref{sec:choices0}, we now consider proper Korn-type inequalities with $\mu$-trace terms on the right-hand side. For future reference, we single out the following elementary lemma; as usual, $\dim_{\mathcal{H}}$ denotes the Hausdorff dimension (see, e.g., \cite[Chpt.~2]{EvansGariepy}). 
\begin{lemma}\label{lem:vanishcap}
Let $\Omega\subset\R^{n}$ be open and bounded, and let $\mathbf{f}\colon\Omega\to\R^{n}$ be a continuous function.
\begin{enumerate}
\item\label{item:vanish1A} Suppose that, in addition,    $\mathbf{f}\colon\Omega\to\R^{n}$ is injective, and let $\mu\in\mathcal{M}^{+}(\Omega)$ be non-trivial and non-atomic. Then we have 
\begin{align*}
\int_{\Omega}|\mathbf{f}|\dif\mu = 0 \;\; \Longrightarrow \;\; \mathbf{f}\equiv 0. 
\end{align*}
\item\label{item:vanish2A} Suppose that $\mathbf{f}$  vanishes at most on a set of the form $\Omega\cap\mathfrak{S}$, where $\mathfrak{S}\subset\R^{n}$ is a closed set with $\dim_{\mathcal{H}}(\mathfrak{S})\leq \lambda\in[0,n]$.   Moreover, let $\mu\in\mathcal{M}_{c}^{+}(\Omega)$ be a  non-trivial measure such that 
\begin{align}\label{eq:tscherpelcondo}
\dim_{\mathcal{H}}(U)\leq\lambda \;\;\Longrightarrow  \;\;\mu(U)=0\qquad\text{for all Borel sets $U\subset\Omega$}. 
\end{align}
Then we have that 
\begin{align*}
\int_{\Omega}|\mathbf{f}|\dif\mu = 0 \;\; \Longrightarrow \;\;\mathbf{f}\equiv 0. 
\end{align*}
\end{enumerate}
\end{lemma}

\begin{proof}
In proving \ref{item:vanish1A} and \ref{item:vanish2A}, we may assume that $\mathbf{f}$ has at least one zero in $\Omega$. Indeed, if not, then $\mu(\Omega)>0$ trivially implies that 
\begin{align*}
\int_{\Omega}|\mathbf{f}|\dif\mu>0, 
\end{align*}
and so there is nothing to prove. We denote by $N$ the set of all zeros of $\mathbf{f}$ in $\Omega$. In the setting of \ref{item:vanish1A}, $N=\{x_{0}\}$ for some suitable $x_{0}\in\Omega$, and we then put $\mathfrak{S}\coloneqq N$. Hence, in both \ref{item:vanish1A} and \ref{item:vanish2A}, $\mathfrak{S}$ is closed. Now suppose that  $\mathbf{f}\not\equiv 0$. Let $K\subset\Omega$ be an arbitrary compact subset and define, for $r>0$, 
\begin{align*}
\mathscr{U}_{r}(\Omega\cap\mathfrak{S})\coloneqq \{x\in\R^{n}\colon\;\mathrm{dist}(x,\Omega\cap\mathfrak{S})<r\},
\end{align*}
which is open. Hence, for each $r>0$, $K\setminus\mathscr{U}_{r}(\Omega\cap\mathfrak{S})$ is compact. By continuity of $\mathbf{f}$, there exists $\varepsilon_{K,r}>0$ such that $|\mathbf{f}|\geq\varepsilon_{K,r}$ on 
 on $K\setminus\mathscr{U}_{r}(\Omega\cap\mathfrak{S})$. Hence, 
 \begin{align*}
\int_{\Omega}|\mathbf{f}|\dif\mu = 0 \;\;\Longrightarrow \;\;\varepsilon_{K,r}\mu(K\setminus\mathscr{U}_{r}(\Omega\cap\mathfrak{S}))\leq \int_{\Omega}|\mathbf{f}|\dif\mu = 0 \;\;\Longrightarrow \;\;\mu(K\setminus\mathscr{U}_{r}(\Omega\cap\mathfrak{S})) = 0.
 \end{align*}
 Next, note that $\mathbbm{1}_{\mathscr{U}_{r}(\Omega\cap\mathfrak{S})}\to 0$ $\mu$-a.e.~in $\Omega$. Indeed, let $x\in\Omega\setminus\mathfrak{S}$. Then $\mathrm{dist}(x,\mathfrak{S})>0$ by the closedness of $\mathfrak{S}$, and so there exists $r_{0}>0$ such that $\ball_{r_{0}}(x)\subset\Omega\setminus\mathfrak{S}$. Thus, $\mathbbm{1}_{\Omega\cap\mathscr{U}_{r}(\Omega\cap\mathfrak{S})}(x)\to 0$ as $r\searrow 0$. In each of the cases \ref{item:vanish1A} and \ref{item:vanish2A}, the underlying hypotheses ensure that $\mu(\Omega\cap\mathfrak{S})=0$.  Therefore, $\mu(\Omega\cap\mathscr{U}_{r}(\Omega\cap\mathfrak{S}))\to$ 0 as $r\searrow 0$ by dominated convergence. In conclusion, 
 \begin{align*}
    \mu(K) \leq  \lim_{r\searrow 0}(\mu(K\setminus\mathscr{U}_{r}(\Omega\cap\mathfrak{S})) + \mu(\Omega\cap\mathscr{U}_{r}(\Omega\cap\mathfrak{S})) = 0, 
 \end{align*}
 and so $\mu(K)=0$ for any compact subset $K\subset\Omega$. We then choose a sequence $(K_{j})\subset\Omega$ of compact sets with $K_{j}\nearrow\Omega$.  Dominated convergence yields 
\begin{align*}
\mu(\Omega)=\lim_{j\to\infty}\mu(K_{j})=0. 
\end{align*}
However, since $\mu\in\mathcal{M}^{+}(\Omega)$, we have $\mu(\Omega)>0$, and thus we arrive at a contradiction. Hence, $\mathbf{f}\equiv 0$, and the proof is complete.  
 \end{proof}
Based on the previous lemma, we now turn to specific differential operators.
\begin{example}[Deviatoric gradient]\label{ex:devgradmu} Let $\Omega\subset\R^{n}$ be open, bounded and connected with Lipschitz boundary. For given $1<p<q<\infty$, let $\mu\in\mathcal{M}_{c}^{+}(\Omega)$ be non-trivial and non-atomic such that its trivial extension $\overline{\mu}$ to $\R^{n}$ satisfies \eqref{eq:capcondo1} with $k=1$. By Example \ref{ex:diffops}\ref{item:devgrad}, each element of $\ker(\nabla^{D};\Omega)$ is of the form $\bm{\rho}(x)=\alpha x+b$ for some $\alpha \in\R$ and $b\in\R^{n}$. Now let $\bm{\rho}(x)= \alpha x+b$ be such that $ \alpha\neq 0$, whereby  $\bm{\rho}$ is injective and continuous. Considering the seminorm $\vertiii{\cdot}$ as in \eqref{eq:chiefseminormSection4}, we use Lemma \ref{lem:vanishcap}\ref{item:vanish1A} to conclude $\bm{\rho}\equiv 0$, provided that $\vertiii{\bm{\rho}}=0$. 
If $ \alpha =0$ and so $\bm{\rho}\equiv b$, then $\vertiii{\bm{\rho}}=0$ directly yields $\bm{\rho}\equiv 0$ by virtue of $\mu(\Omega)>0$. 
In conclusion, $\vertiii{\cdot}$ is a norm on $\ker(\nabla^{D};\Omega)$. Since $\nabla^{D}$ is $\mathbb{C}$-elliptic by Example \ref{ex:diffops}, Korn's first inequality \eqref{eq:KornJohn} and so Theorem \ref{thm:variant} directly yield    
\begin{align}\label{eq:devgradmutrace}
    \|\mathbf{u}\|_{\sobo^{1,p}(\Omega)}\lesssim \int_{\Omega}|\mathrm{tr}_{\mu}(\mathbf{u})|\dif\mu + \|\nabla^{D}\mathbf{u}\|_{\lebe^{p}(\Omega)}\qquad\text{for all}\;\mathbf{u}\in\sobo^{1,p}(\Omega;\R^{n}). 
\end{align}
Note that, if 
\begin{itemize}
\item $1<p<n$, then \eqref{eq:capcondo1} already implies that $\mu$ is non-atomic (recall that $n\geq 2$). Hence, \eqref{eq:devgradmutrace} then holds under the same conditions as required in Example \ref{ex:gradient}.
\item If $p>n$, then $\mu$ is still allowed to be atomic in the inequality \eqref{eq:traceineq1} from Example \ref{ex:gradient},  whereas this is forbidden for the inequality \eqref{eq:devgradmutrace} that involves  $\nabla^{D}$.
\end{itemize}
Moreover, we may replace the trace integral in \eqref{eq:devgradmutrace} by $\|\mathrm{tr}_{\mu}(\mathbf{u})\|_{\lebe^{q}(\Omega;\mu)}$; by our discussion in Section \ref{sec:mutraces}, the resulting inequality is still  meaningful. However, note that \eqref{eq:devgradmutrace} is stronger.
\end{example}
\begin{example}[Symmetric  gradient, $n=2$]\label{ex:mutracesymgrad1} We let $n=2$ and assume that $\Omega$ and $\mu\in\mathcal{M}_{c}^{+}(\Omega)$ satisfy the same assumptions as in the preceding example for some $1<p<q<\infty$. By Example \ref{ex:diffops}\ref{item:symgrad}, any element of $\ker(\sg;\Omega)$ is of the form $\bm{\rho}(x)=Ax+b$ with some $A\in\R_{\mathrm{skew}}^{2\times 2}$ and $b\in\R^{2}$. With the seminorm $\vertiii{\cdot}$ as in \eqref{eq:chiefseminormSection4}, suppose that $\bm{\rho}$ satisfies $\vertiii{\bm{\rho}}=0$. If $A\neq 0$, then $\bm{\rho}$ is injective; note that this conclusion precisely holds in $n=2$ dimensions. Again employing Lemma \ref{lem:vanishcap}\ref{item:vanish1A}, we conclude that $\bm{\rho}\equiv 0$. If $A=0$, then $\vertiii{\bm{\rho}}=0$ immediately implies $\bm{\rho}=0$. In particular, we have 
\begin{align}\label{eq:symgradmutraces2d}
\|\mathbf{u}\|_{\sobo^{1,p}(\Omega)}\lesssim \int_{\Omega}|\mathrm{tr}_{\mu}(\mathbf{u})|\dif\mu + \|\sg(\mathbf{u})\|_{\lebe^{p}(\Omega)}\qquad\text{for all}\;\mathbf{u}\in\sobo^{1,p}(\Omega;\R^{2}).
\end{align}
The discussion at the end of Example \ref{ex:devgradmu} (with $n=2$) also applies to the present situation.
\end{example}
For the symmetric gradient, the case $n=2$ is special. Namely, if $n\geq 3$, elements of $\ker(\varepsilon;\R^{n})$ are not injective in general if $n\in 2\mathbb{N}+2$, and are never injective if $n\in 2\mathbb{N}+1$: If $A\in\R_{\mathrm{skew}}^{n\times n}$ and $n\in2\mathbb{N}+1$, then we have 
\begin{align*}
\det(A)=\det(A^{\top})=\det(-A)=(-1)^{n}\det(A) \stackrel{n\in2\mathbb{N}+1}{=} -\det(A)\Longrightarrow  \det(A)=0. 
\end{align*}
\begin{example}[Symmetric gradient, $n= 3$]\label{ex:symgradn=3}
We let $n=3$, but note that the following also applies to the higher dimensional case with the obvious modifications. 
If $n=3$, then we have $\dim(\ker(A))\in\{1,3\}$ for any $A\in\R_{\mathrm{skew}}^{3\times 3}$. In what follows, we let $\bm{\rho}(x)=Ax+{b}$ with $A\in\R_{\mathrm{skew}}^{n\times n}$ and ${b}\in\R^{n}$. We consider the condition 
\begin{align*}
\vertiii{\bm{\rho}}\coloneqq \int_{\Omega}|\bm{\rho}|\dif\mu \stackrel{!}{=} 0. 
\end{align*}
Now, if $\mu=\mathscr{H}^{1}\mres\ell$ for some line $\ell$, then $\vertiii{\cdot}$ is not a norm on the nullspace; this is the case if $\dim(\ker(A))=1$. Similarly, if $0\leq s<1$ and $\mu=\mathscr{H}^{s}\mres U$ with some $s$-dimensional set $U$ contained in a line $\ell$, then $S_{1}$ is not a norm. In consequence, a Korn inequality with $\mu$-trace terms fails in this case.  

Conversely, if $\ell_{1}$ and $\ell_{2}$ are two non-collinear lines and $\mu=\mathscr{H}^{1}\mres (\ell_{1}\cup\ell_{2})$, then $\vertiii{\cdot}$ is a norm on the nullspace of $\varepsilon$ indeed. For instance, if $p>3$, then one has 
\begin{align*}
\|\mathbf{u}\|_{\sobo^{1,p}(\Omega)}\lesssim \|\mathrm{tr}_{\mu}(\mathbf{u})\|_{\lebe^{1}(\Omega)} + \|\sg(\mathbf{u})\|_{\lebe^{p}(\Omega)}\qquad\text{for all}\;\mathbf{u}\in\sobo^{1,p}(\Omega;\R^{3})
\end{align*}
if $\Omega\subset\R^{3}$ is open and bounded with Lipschitz boundary; note that, in this case, $\sobo^{1,p}(\Omega;\R^{3})\hookrightarrow\hold(\overline{\Omega};\R^{3})$, whereby the trace integral is well-defined indeed. Other combinations of integrabilities with respect to $p$ and $q$ can be obtained as in the previous examples. 
\end{example}

\begin{example}[Deviatoric symmetric gradient, $n\geq 3$]\label{ex:devgrad}
As discussed in Example \ref{ex:devsymgradcomplications}, the case of the deviatoric symmetric gradient allows for a larger variety of the underlying nullspaces; in fact, if $n=3$, the nullspace is an intersection of three algebraic varieties. This, however, might lead to curved objects, and so a treatment as in the previous examples (where the corresponding set where nullspace elements vanish were affine subspaces) is impossible. In particular, in general, there is an interplay between the operators, the integrabilities and the geometry of the supports of the underlying measures. 

\end{example}

 \section{Symbolic verification of the norm condition} \label{sec:sym-comput}

Let us recall, that the validity of condition~\ref{item:Korn1B}~\ref{item:Korn1}  in Theorem~\ref{thm:korn-gen}, namely that $\vertiii{\cdot}$ is a norm on $\ker(\mathbb{A};\mathbb{R}^n)|_{\Omega}$ is strongly linked to the geometry of $\partial \Omega$. 
In some cases it is known, for which domains this is violated, see Proposition ~\ref{prop:devgradpartial} and~Example \ref{prop:partialtracesymgrad}, but in others it is far from clear. 

Here we present a way of testing for some special cases, whether this condition is satisfied, by means of symbolic calculations. 
We assume that $\Omega \subset \mathbb{R}^n$ is a open bounded connected set, and that its boundary $\partial \Omega$ is represented by a parametrisation $r \colon \Theta \to \mathbb{R}_{>0}$ for $\Theta = [0,\pi] \times [0,2\pi)^{n-2}$ in spherical coordinates $(\theta, r)$.  
Furthermore, we consider the special case $V = \mathbb{R}^n$, $\Gamma = \partial \Omega$, $T(\bu) = \operatorname{tr}(\bu) \cdot \nu_{\partial \Omega}$ and for $n \in \{2,3\}$. 
Furthermore, for simplicity we assume that  we know $K \in \mathbb{N}$ such that $\ker(\mathbb{A};\mathbb{R}^n) \subset \mathcal{P}_{K}(\mathbb{R}^n;\mathbb{R}^n)$. In this case we have 
\begin{align}
\label{eq:kerneleq}
S_{\mathbb{A},K} \coloneqq \ker(\mathbb{A};\mathbb{R}^n) \cap \mathcal{P}_K(\mathbb{R}^n;\mathbb{R}^n) = \ker(\mathbb{A};\mathbb{R}^n). 
\end{align}
 Note that $K\in\mathbb{N}$ such that $\ker(\mathbb{A};\mathbb{R}^n)\subset \mathcal{P}_K(\mathbb{R}^n;\mathbb{R}^n)$ is known in many cases; see Example~\ref{ex:diffops}. On an abstract level, $K\in\mathbb{N}$ can be obtained by use of the Buchberger algorithm from computational algebraic geometry \cite{Buchberger}; this is due to the fact that $K$ stems from the Hilbert Nullstellensatz \cite{Kalamajska,Smith}; indeed, it is the $\mathbb{C}$-ellipticity that gives access to such techniques. In practice, however, $K$ is often known, which is why we assume it to be given. 

The pseudocode in Algorithm~\ref{alg:sym} describes our approach in more detail: it consists of determining the kernel $\ker(\mathbb{A};\mathbb{R}^n)$ symbolically, and then testing whether or not, there are non-trivial kernel elements, which satisfy $T(\bu) = 0$ on $\Gamma = \partial \Omega$. For the tests we choose  
\begin{align}\label{def:norm-Theta}      
    \vertiii{\bu}_{\overline \Theta} \coloneqq \sum_{\theta \in \overline{\Theta}} |T \bu (x(\theta))|
\end{align}
for some given discrete set $\overline{\Theta}\subset \Theta$. 
This is useful for the following reason:  if $\bu \mapsto \vertiii{T(\bu)}_{\overline \Theta}$ is a norm on $\ker(\mathbb{A};\mathbb{R}^n) \subset C(\mathbb{R}^n;\mathbb{R}^n)$, then so is, e.g., 
$\vertiii{\bu} \coloneqq \norm{T(\bu)}_{L^p(\Gamma)}$. 
\smallskip

\begin{algorithm}\label{alg:sym}
\caption{Symbolic test for norm on $\operatorname{ker}(\mathbb{A};\mathbb{R}^n) \cap \mathcal{P}_K$ }
\begin{algorithmic}[1]
\Require $K \in \mathbb{N}$ polynomial degree, $n \in  \{2,3\}$,  
tensor $\texttt{A} \in \mathbb{R}^{n \times n \times n \times n}$ such that
\[
(\mathbb{A} u)_{i,j}(x) = \sum_{k,l = 1}^n \texttt{A}_{i,j,k,l}\partial_{x_l} u_k(x) 
\]
discrete $\overline \Theta \subset  \Theta \coloneqq  [0,\pi] \times [0,2\pi)^{n-2} $, 
$r \colon \Theta \to \mathbb{R}_{>0}$ parametrisation of $\Gamma$ in spherical coordinates

\Ensure Decide whether $\norm{\bu}_{\overline \Theta}$ as in \eqref{def:norm-Theta} defines a norm on 
 $   S_{\mathbb{A},K} \coloneqq  \mathcal{P}_K(\mathbb{R}^n;\mathbb{R}^n) \cap \ker(\mathbb{A};\mathbb{R}^n)$ and return $\bm{\rho} \in S_{\mathbb{A},K}\setminus \{0\}$ such that $\bm{\rho} \cdot \nu |_{\Gamma} = 0$, if it exists.  
\vspace{1em}

\State Symbolically compute $\mathbb{A} \bm{\rho}_{\alpha}(x)$ 
for a general polynomial $\bm{\rho}_{\alpha} \in \mathcal{P}_K(\mathbb{R}^n;\mathbb{R}^n)$ 
\Statex with variables $x \in \mathbb{R}^n$ and coefficients $\alpha = \alpha(\bm{\rho}) \in \mathbb{R}^{m}$.
\Comment{determine  $\mathbb{A}p$}
\vspace{0.7em}

\State  Symbolically determine  $S_{\mathbb{A},K}$ via 
  $  S_{\mathbb{A},K} \coloneqq \{ \bm{\rho} \in \mathcal{P}_K(\mathbb{R}^n;\mathbb{R}^n) \colon
     \mathbb{A} \bm{\rho}(y) = 0 \; \forall y \in Y\}$ 
\Statex for suitably chosen discrete set $Y \subset \mathbb{R}^n$. 
\Comment{determine  $K_{\mathbb{A},r} $}
\vspace{0.7em}

    \State For $x \in \Gamma$ symbolically represent $x = x(\theta)$ and $\nu = \nu(x(\theta))$ for $\theta \in \Theta$. 
    \Comment{boundary}
\vspace{0.7em}

\State For general $\bm{\rho}_{\alpha} \in S_{\mathbb{A},K}$ 
symbolically  determine      
\Comment{boundary test}
\begin{align*}
C_{\overline \Theta} \coloneqq \{\alpha  \in \mathbb{R}^m \colon \bm{\rho}_{\alpha}(x(\theta)) \cdot \nu(x(\theta)) = 0 \quad \text{for all } \theta \in \overline{\Theta} \}
\end{align*}
    \If{$ C_{\overline \Theta}= \{0\}$}  \Comment{only trivial functions}
        \State Conclusion~\ref{itm:A1}: $\vertiii{\cdot}_{\overline \Theta}$  is norm on $S_{\mathbb{A},K}$  \Comment{test positive}
        
    \Else         \Comment{non-trivial solution exists}

        \State Symbolically compute $\theta \mapsto g_{\alpha} \coloneqq \bm{\rho}_{\alpha}(x(\theta)) \cdot \nu(x(\theta))$ for general $\alpha \in C_{\overline \Theta}\setminus \{0\}$ 
        \If{$g_{\alpha} \equiv 0$ on $\Theta$}
           \State Conclusion~\ref{itm:A2}: there is  $\bm{\rho} \in S_{\mathbb{A},K} \setminus \{0\}$ with $\bm{\rho} \cdot \nu|_{\Gamma} = 0$  \Comment{test negative}
        \Else
          \State Conclusion~\ref{itm:A3}
          \Comment{test inconclusive}
        \EndIf
    \EndIf
    
\end{algorithmic}
\end{algorithm}

\noindent Possible outcomes of the Algorithm are as follows: 
\begin{enumerate}[label = (A\arabic*)]
    \item \label{itm:A1} $\vertiii{\bu}_{\overline \Theta}$ is a norm on $\ker(\mathbb{A};\mathbb{R}^n) \subset \mathcal{P}_{K}(\mathbb{R}^n;\mathbb{R}^n)$, and hence also other seminorms relying on point evaluations including the ones on $\overline{\Theta}$ are norms. 
    \item \label{itm:A2} There is a non-trivial kernel element $\bm{\rho} \in \ker(\mathbb{A};\mathbb{R}^n) \subset \mathcal{P}_{K}(\mathbb{R}^n;\mathbb{R}^n)$ such that $T(\bm{\rho}) = 0$ on $\partial \Omega$; in this case no norm $\vertiii{\cdot}$ only relying on values of $T(\bu) = \operatorname{tr}(\bu) \cdot \nu_{\partial \Omega} $ is available, and $\bm{\rho}$ serves as certificate. 
    \item \label{itm:A3} \ref{itm:A1} and \ref{itm:A2} do not hold. In this case $\vertiii{\cdot}_{\overline \Theta}$ is not a norm, but we cannot say anything about other seminorms $\vertiii{\cdot}$ possibly including more point values. In that sense the test is inconclusive and one may want to repeat it for a larger subset $\overline \Theta \subset \Theta$. 
\end{enumerate}

Matlab code~\cite{Matlab2023} for the symbolic computation will be made available together with the published version of the paper. A possible outcome of the code is as follows: 

\begin{example} 
We denote by $M(\mathbb{A};\Omega)$ the subset of the kernel $\ker(\mathbb{A};\mathbb{R}^n)$, which has zero normal trace on $\Omega$.
We have tested the code with the following instances:
\begin{itemize}
    \item For $n = 2$, $K= 1$ and $r(\theta) \equiv 1$, i.e., $\Omega_1 = B_1(0)$
    we have 
    \begin{align*}
        M(\nabla^D;B_1(0)) & =\{0\}\\
        M(\sg;B_1(0)) &= \{x \mapsto \beta ( -x_2, x_1)^\top, \beta \in \mathbb{R} \}. \end{align*}
    \item For $n = 2$,  $K= 1$  and for domain $\Omega_2$ with boundary parametrised by $r(\theta) = (2 + \sin(2\theta))$ we obtain with choosing $\overline \Theta$ sufficiently large ($\#  \overline \Theta \geq 6$)
            \begin{align*}
        M(\nabla^D;\Omega_2)  =\{0\}  =  M(\sg;\Omega_2).
        \end{align*}
  \item For $n = 3$ (for $K = 2$, $\# \overline \Theta = 4^2$) and $r(\theta) \equiv 1$, i.e., $\Omega_3 = B_1(0)$ we have 
        \begin{align*}
        M(\nabla^D;B_1(0)) & =\{0\}, 
        \\
        M(\sg;B_1(0)) & \left\{x \mapsto \beta_1 
        \begin{pmatrix} -x_3\\0\\ x_1\end{pmatrix} + \beta_2 
        \begin{pmatrix}-x_2\\x_1\\ 0 \end{pmatrix} + \beta_3 \begin{pmatrix}0\\ -x_3\\ x_2\end{pmatrix}, \beta \in \mathbb{R}^3 \right\}\\
        & = M(\sg^D;B_1(0));
        \end{align*}
        \item For $n = 3$ and for domain $\Omega_4$, with boundary parametrised by $r(\theta) = 2 + \sin(2 \theta_1) \sin(3\theta_2) $, 
          \begin{align*}
        M(\nabla^D;\Omega_4) =
        M(\sg;\Omega_4) = 
        M(\sg^D;\Omega_4) = \{0\},
        \end{align*}
        for $K = 2$ and $\# \overline \Theta = 6^2$.
\end{itemize} 
\end{example}

The code can easily be adapted to other boundary conditions, e.g., tangential boundary conditions,  boundary conditions on $\Gamma \subsetneq \partial \Omega$ and to other parametrisations of $\partial \Omega$, e.g., using cylindrical coordinates.

\section*{Acknowledgments}
F.G. acknowledges financial support through the Hector Foundation, Project-ID FP 626/21. The work by T.T. was supported by  the German Research Foundation (DFG) via grant TRR 154, subproject C09, Project Number 239904186.

\end{document}